\documentclass[12pt]{article}
\usepackage{fancyhdr}
\usepackage{setspace}
\usepackage{enumitem}
\setlist{nolistsep}
\usepackage[mathscr]{eucal}

\usepackage{pifont}
\usepackage{amssymb}
\usepackage{amsmath}
\usepackage{graphicx}
\usepackage[usenames,dvipsnames]{color}
\usepackage{stmaryrd} \newcommand{\rbose}{\rrbracket}\newcommand{\lbose}{\llbracket}\newcommand{\rround}{\rrparenthesis}\newcommand{\lround }{\llparenthesis}

\newtheorem{theorem}{Theorem}[section]
\newtheorem{lemma}[theorem]{Lemma}
\newtheorem{corollary}[theorem]{Corollary}

\newtheorem{definition}[theorem]{Definition}

\newtheorem{result}[theorem]{Result}
\newenvironment{proof}{\noindent{\bf Proof}\hspace{0.5em}}
    { \null \hfill $\square$ \par}

\newcommand{\VC}{\V(\Bo{\C})}

\newcommand{\VCBB}{\V([\C])}

\newcommand{\SI}{{\textup{[S1]}}}
\newcommand{\SII}{{\textup{[S2]}}}
\newcommand{\SIII}{{\textup{[S3]}}}
\newcommand{\TI}{{\textup{[T1]}}}
\newcommand{\TII}{{\textup{[T2]}}}
\newcommand{\TIII}{{\textup{[T3]}}}
\newcommand{\TIV}{{\textup{[T4]}}}
\newcommand{\EI}{{\textup{[E1]}}}
\newcommand{\EII}{{\textup{[E2]}}}
\newcommand{\EIII}{{\textup{[E3]}}}
\newcommand{\EIV}{{\textup{[E4]}}}

    \newcommand{\Fq}{\mathbb F_{q}}
\newcommand{\Fqqq}{\mathbb F_{q^3}}

\newcommand{\Fqqqqqq}{\mathbb F_{q^6}}

\newcommand{\plus}{{\raisebox{.2\height}{\scalebox{.5}{\textup{+}}}}}
\newcommand{\Cplus}{\C^{{\rm \pmb\plus}}}
\newcommand{\si}{\Sigma_\infty}
\renewcommand{\O}{\mathcal O}

\newcommand\B{{\cal B}}

\newcommand{\X}{\mathcal X}

\newcommand{\R}{\mathcal R}
\newcommand\C{{\cal C}}

\newcommand\V{{\cal V}}

\newcommand{\IBB}{{\cal I}_{\mbox{\scriptsize\sf BB}}}
\newcommand{\IB}{{\cal I}_{\mbox{\scriptsize\sf Bose}}}
\newcommand{\XT }{{X}_{\mbox{\raisebox{.01\height}{\scalebox{.8}{{\tiny \!$T$}}}}}}

\newcommand{\Pit}{\Gamma}

\newcommand\A{{\cal A}}

\newcommand\N{{\cal N}}

\renewcommand{\S}{\mathcal S}
\renewcommand{\P}{\mathcal P}
\newcommand{\K}{\mathcal K}

\newcommand{\Q}{\mathscr Q}

\newcommand{\li}{\ell_\infty}

\newcommand{\takeaway}{\backslash}
\renewcommand\setminus{\backslash}
\newcommand{\st}{\,|\,}

\newcommand\PGammaL{{\textup{P}\Gamma \textup{L}}}
\newcommand\PGL{{\rm PGL}}

\newcommand\PG{{\rm PG}}

\renewcommand{\star}{{^{\mbox{\tiny\ding{73}}}}}
\newcommand\Bo[1]{\mbox{$\lbose#1\rbose$}}

\newcommand{\gstar}{g^{\mbox{\tiny\ding{72}}}}
\newcommand{\gqstar}{{{g}^{\mbox{\tiny\ding{72}}}}^q}
\newcommand{\gqqstar}{{{g}^{\mbox{\tiny\ding{72}}}}^{q^2}}

\newcommand{\sistar}{{\Sigma}^{\mbox{\tiny\ding{73}}}_{\mbox{\raisebox{.4\height}{\scalebox{.6}{$\infty$}}}}}
\newcommand{\sistarstar}{{\Sigma}^{\mbox{\tiny\ding{72}}}_{\mbox{\raisebox{.4\height}{\scalebox{.6}{$\infty$}}}}}

\newcommand{\Pirstar}{{\Pi}^{\mbox{\tiny\ding{73}}}_r}
\newcommand{\Pirstarstar}{{\Pi}^{\mbox{\tiny\ding{72}}}_r}
\newcommand{\Pigstar}{{\Pi}^{\mbox{\tiny\ding{73}}}_g}
\newcommand{\Pigstarstar}{{\Pi}^{\mbox{\tiny\ding{72}}}_g}
\newcommand{\Pithreestar}{{\Pi}^{\mbox{\tiny\ding{73}}}_3}

\newcommand{\Sigmarstarstar}{{\Sigma}^{\mbox{\tiny\ding{72}}}_r}
\newcommand{\Sigmaqstar}{{\Sigma}^{\mbox{\tiny\ding{73}}}_{6,q}}

\newcommand{\Qonestar}{{\Q}^{\mbox{\tiny\ding{73}}}_1}
\newcommand{\Qninestar}{{\Q}^{\mbox{\tiny\ding{73}}}_9}
\newcommand{\Qonestarstar}{{\Q}^{\mbox{\tiny\ding{72}}}_1}
\newcommand{\Qninestarstar}{{\Q}^{\mbox{\tiny\ding{72}}}_9}

\renewcommand{\star}{{^{\mbox{\tiny\ding{73}}}}}
\newcommand{\blackstar}{{^{\mbox{\tiny\ding{72}}}}}

\newcommand{\ellbstarstar}{{\ell^{\mbox{\tiny\ding{72}}}_b}}

\newcommand{\Pstar}{{[P]^{\mbox{\tiny\ding{73}}}}}

\newcommand{\Kstar}{{\K^{\mbox{\tiny\ding{73}}}}}
\newcommand{\Kstarstar}{{\K^{\mbox{\tiny\ding{72}}}}}
\newcommand{\Nstar}{{\N^{\mbox{\tiny\ding{73}}}}}

\newcommand{\VCstar}{\V(\Bo{\C})^{\!\mbox{\tiny\ding{73}}}}
\newcommand{\VCBBstar}{\V([\C])^{\!{\mbox{\tiny\ding{73}}}}}
\newcommand{\VCstarstar}{\V(\Bo{\C})^{\!\mbox{\tiny\ding{72}}}}
\newcommand{\VCBBstarstar}{\V([\C])^{\!\mbox{\tiny\ding{72}}}}

\newcommand{\Nthreestar}{{\N}^{\mbox{\tiny\ding{73}}}_3}

\newcommand{\Nfourstar}{{\N}^{\mbox{\tiny\ding{73}}}_4}

\newcommand{\Nsixstar}{{\N}^{\mbox{\tiny\ding{73}}}_6}
\newcommand{\Nsixstarstar}{{\N}^{\mbox{\tiny\ding{72}}}_6}
\newcommand{\Ntwostarstar}{{\N}^{\mbox{\tiny\ding{72}}}_2}

\newcommand{\Qstar}{{\Q^{\mbox{\tiny\ding{73}}}}}

\newcommand\conjB[1]{{  #1^{\cpi}}}
\newcommand\conjBsq[1]{{  #1^{\mathsf c^2_{\pi}}}}

  \newcommand{\cpi}{\mathsf c_\pi}
   \newcommand{\ocpi}{\bar{\mathsf c}_\pi}

 \newcommand{\cpis}{\mathsf c^2_\pi}
  \newcommand{\ocpis}{\bar{\mathsf c}^2_\pi}
 \newcommand{\ocpic}{\bar{\mathsf c}^3_\pi}
  \newcommand{\ocpif}{\bar{\mathsf c}^4_\pi}
   \newcommand{\ocpiff}{\bar{\mathsf c}^5_\pi}
  \newcommand{\ocpii}{\bar{\mathsf c}^i_\pi} 
  
  \newcommand{\cb}{\mathsf c_b}
   \newcommand{\ocb}{\bar{\mathsf c}_b}

\newcommand{\Label}{\label}

\newcommand{\Cplusplus}{\C^{{\rm \plus\!\plus}}}
\newcommand{\barCplusplus}{{\bar \C}^{{\rm \plus\!\plus}}}
\newcommand{\liplusplus}{\ell^{{\rm \plus\!\plus}}_\infty}

\begin{document}
%


\title{A characterisation of $\Fq$-conics of $\PG(2,q^3)$}
%
\author{S.G. Barwick, Wen-Ai Jackson and Peter Wild}
\date{\today}
\maketitle
%
%
%
%
Keywords: Bruck-Bose representation, Bose representation,   $\Fq$-subplanes, $\Fq$-conics, normal rational curves.
AMS code: 51E20
%



\begin{abstract}
This article considers an $\Fq$-conic contained in an $\Fq$-subplane  of $\PG(2,q^3)$, and shows that it corresponds to a normal rational curve in the Bruck-Bose representation in $\PG(6,q)$. 
The main result characterises which normal rational curves of $\PG(6,q)$ correspond via the Bruck-Bose representation to $\Fq$-conics of $\PG(2,q^3)$. The normal rational curves   of interest are called 3-special, a property which describes how the extension of the normal rational curve meets  the transversal lines of the regular 2-spread of the Bruck-Bose representation. 
The proof uses geometric arguments that exploit the interaction between the Bruck-Bose representation  of $\PG(2,q^3)$ in $\PG(6,q)$, and the Bose representation  of $\PG(2,q^3)$ in $\PG(8,q)$. 
%
\end{abstract}

\section{Introduction}

An $\Fq$-plane of $\PG(2,q^n)$ is a subplane of $\PG(2,q^n)$ which has order $q$. 
An $\Fq$-conic is a non-degenerate conic in an $\Fq$-subplane, and so is projectively equivalent to a non-degenerate conic of $\PG(2,q)$.
The representation of $\Fq$-conics in the Bruck-Bose representation when $n=2$ was first looked at in \cite{quinn-conic}; and the following complete characterisation was given in \cite{BJW2}. 

\begin{result} An $\Fq$-conic in $\PG(2,q^2)$ corresponds in the $\PG(4,q)$ Bruck-Bose representation to a 2-special normal rational curve. Further, a normal rational curve in $\PG(4,q)$ corresponds to an $\Fq$-conic in $\PG(2,q^2)$ if and only if it is 2-special.
\end{result}

The notion of 2-special describes how the extension of a normal rational curve meets the transversal lines of the Bruck-Bose spread. Let $\PG(4,q)$ have hyperplane at infinity $\si$ and regular 1-spread $S$ in $\si$.  
Let $\N_r$ be an
$r$-dimensional normal rational curve, $ r\leq 4$, that is not contained in $\si$, then $\N_r$ meets $\si$ in $r$ points $P_1,\ldots,P_r$, possibly repeated, possibly in an extension. We say a point $P$ in an extension of $\si$ has weight $w(P)=1$ if $P$ lies on an extended transversal line of the regular 1-spread $S$; otherwise, $w(P)=2$.
The normal rational curve $\N_r$ is called \emph{ $2$-special} with respect to $S$ if $w(P_1)+\ldots+w(P_r)=4$. 

 This article looks at the representation of $\Fq$-conics in the Bruck-Bose representation when $n=3$. The main result is the following complete characterisation of $\Fq$-conics of $\PG(2,q^3)$ in the $\PG(6,q)$ Bruck-Bose representation.
 
  \begin{theorem}\Label{thm-main}  Let $\PG(6,q)$ have hyperplane at infinity $\si$ and regular 2-spread $\S$ in $\si$. An $r$-dimensional normal rational curve $\N_r$, $r\leq 6$,  corresponds via the Bruck-Bose representation to an $\Fq$-conic of $\PG(2,q^3)$ if and only if $\N_r$ is $3$-special. 
\end{theorem}

 The notion of 3-special  normal rational curves is defined in Definitions \ref{def-good} and \ref{def-2spec}. It relates to how the extension of the normal rational curve meets the three transversal lines of the Bruck-Bose regular 2-spread, and is a non-trivial generalisation of the idea of 2-special  normal rational curves in $\PG(4,q)$.

This characterisation is proved using geometric arguments that exploit the interaction between the Bruck-Bose representation in $\PG(6,q)$ and the Bose representation in $\PG(8,q)$, and builds on previous work by the authors on these representations.

The article is set out as follows. 
We look at  $\PG(2,q^3)$ using several different models, namely the Bruck-Bose representation in $\PG(6,q)$ with the usual conventions; a more precise exact-at-infinity Bruck-Bose representation in $\PG(6,q)$; and the Bose representation in $\PG(8,q)$. 
In Section~\ref{sec-back} we describe our notation which is carefully designed to differentiate between the different  models. For easy reference, the notation  
is summarised     in Section~\ref{sec-notn}.

Section~\ref{sec:Fqconic} begins with some preliminaries on $\Fq$-conics of $\PG(2,q^3)$. 
Result~\ref{Fqconic-cases} outlines eleven different cases for how an $\Fq$-conic sits in relation to the line at infinity $\li$.  
The main result of Section~\ref{sec:Fqconic} is   Theorem~\ref{fqconic-order6}, which 
 uses the interplay between the Bose representation and the Bruck-Bose representation to show that an $\Fq$-conic of $\PG(2,q^3)$ corresponds in the `exact-at-infinity' Bruck-Bose representation in  $\PG(6,q)$ to an irreducible curve of degree $k$ and a linear component which is contained in the hyperplane at infinity. 

In Section~\ref{sec4}, we show that this irreducible curve is a  normal rational curve, and we determine $k$ and describe the linear component  in detail for the eleven different cases. Then, using the  traditional convention for the Bruck-Bose representation, we ignore the linear component at infinity, and deduce that every $\Fq$-conic of $\PG(2,q^3)$ corresponds to a normal rational curve of $\PG(6,q)$.  The eleven cases can then be reduced to the following three cases: an $\Fq$-conic in an $\Fq$-subplane secant to $\li$ corresponds to a non-degenerate conic of $\PG(6,q)$; an $\Fq$-conic in an $\Fq$-subplane tangent to $\li$ corresponds to either a 4-dimensional or 6-dimensional normal rational curve of $\PG(6,q)$; and 
an $\Fq$-conic in an $\Fq$-subplane exterior to $\li$ corresponds to either a 3-dimensional or 6-dimensional normal rational curve of $\PG(6,q)$. In each case, we determine the relationship between the normal rational curve and the transversal lines of the regular 2-spread of the Bruck-Bose representation. 

We then consider the converse, and determine which normal rational curves of $\PG(6,q)$ correspond via the Bruck-Bose representation to $\Fq$-conics of $\PG(2,q^3)$. 
    In  Section~\ref{sec-2spec}, we look at how a normal rational curve  meets the transversal lines of the Bruck-Bose  regular 2-spread, and define $3$-special normal rational curves. In Theorem~\ref{2spec-h3} we classify the $3$-special normal rational curves of $\PG(6,q)$. In Section~\ref{sec-main} we look at each possible $3$-special normal rational curve and show that it in each case it corresponds via the Bruck-Bose representation to an $\Fq$-conic of $\PG(2,q^3)$. This leads to a proof of the main result Theorem~\ref{thm-main}. 
We conclude in Section~\ref{conclude} with a  discussion of the general case.

\section{Preliminaries}\Label{sec-back}

An \emph{$r$-dimensional normal rational curve} in $\PG(n,q)$, $q\geq r$,   is a set of points
lying in an $r$-space which is projectively equivalent to the set $$\{(1,\theta,\ldots,\theta^{r})\st\theta\in\Fq\}\cup\{(0,\ldots,0,1)\},$$ see \cite{HT}. We abbreviate this to an \emph{$r$-dim nrc.} 
We will repeatedly use the geometrical property that an $r$-dim nrc $\N_r$ is a set of $q+1$ points in an $r$-space,  such that no $t+2$ points of $\N_r$ lie in a $t$-space, $t=1,\ldots,r-1$. 
As we work with $6$-dim nrcs in this article, we assume $q\geq6$ throughout.

\subsection{Conjugacy with respect to an  $\Fq$-subplane}\Label{sec-conj}

Let $\Fq$ denote the finite field of prime power order $q$. 
An $\Fq$-subplane  of $\PG(2,q^3)$ is a subplane   of $\PG(2,q^3)$  which has order $q$, 
that is, a subplane which is isomorphic to $\PG(2,q)$. An $\Fq$-subline is a line of an $\Fq$-subplane, that is,  isomorphic to  $\PG(1,q)$. We will define conjugacy with respect to an $\Fq$-subplane or $\Fq$-subline. 

Let $\bar \pi$ be an $\Fq$-subplane of $\PG(2,q^3)$.
 Acting on the points of $\PG(2,q^3)$ is a unique collineation group ${\bar G}_\pi\subseteq \PGammaL(3,q^3)$ which fixes $\bar \pi$ pointwise and has order 3.
We need to distinguish between the two non-identity maps in ${\bar G}_\pi$, and 
as discussed in \cite{BJW3}, we can without loss of generality  write ${\bar G}_\pi=\langle \ocpi\rangle$ with 
 \begin{eqnarray}\label{defcpi}
 \ocpi\colon \PG(2,q^3)&\longrightarrow& \PG(2,q^3) \nonumber \\
 \bar  X=\begin{pmatrix} x\\y\\z\end{pmatrix}&\longmapsto & B \bar X^q  = B  \begin{pmatrix} x^q\\y^q\\z^q\end{pmatrix}\end{eqnarray}
 with $B$ a $3\times3$ non-singular matrix over $\Fqqq$. So ${\ocpi}^3=id$ and  for $\bar X\in\PG(2,q^3)\setminus\bar \pi$, the three  points $\bar X$, $\bar X^{\ocpi}$, $\bar X^{\ocpis}$ are called \emph{conjugate with respect to $\bar \pi$}. Note that  
 $\bar X$, $\bar X^{\ocpi}$, $\bar X^{\ocpis}$ are collinear  if and only if $\bar X$ lies on an extended line of $\bar \pi$.
Moreover,  we can uniquely extend the plane $\PG(2,q^3)$ to $\PG(2,q^6)$, and 
  the collineation $\ocpi\in \PGammaL(3,q^3)$ has a natural extension to a collineation of $\PGammaL(3,q^6)$
acting on points of $\PG(2,q^6)$; we  use the same notation $\ocpi$ for this (extended) collineation. The collineation $\ocpi$ has order 3 when acting on $\PG(2,q^3)$, and order 6 when acting on $\PG(2,q^6)$. 
Under the collineation $\ocpi$, a point $\bar X\in\PG(2,q^6)$ lies in an orbit of size:
 $1$ if  $\bar X\in\bar \pi$;
 $3$ if $\bar X\in\PG(2,q^3)\setminus\bar \pi$;
 $2$ or $6$ if $\bar X\in\PG(2,q^6)\setminus\PG(2,q^3)$.

 Similarly, if $\bar b$ is an $\Fq$-subline of a line $\bar \ell_b$ of $\PG(2,q^3)$, then acting on the points of $\bar \ell_b$ is a unique collineation group $\bar G_b\subseteq \PGammaL(2,q^3)$ of order 3 which fixes $\bar b$ pointwise. 
Moreover, ${\bar G}_\pi$ restricted to acting on $\bar \ell_b$ is isomorphic to $\bar G_b$ if and only if  $\bar b$ is a line of $\bar \pi$. 
Without loss of generality we can write ${\bar G}_b=\langle \ocb \rangle $  where for a point $\bar X\in\bar\ell_b$, 
\begin{eqnarray}
\label{defcb}
\ocb(\bar X)&=&D\bar X^q
\end{eqnarray}
 with $D$ a non-singular matrix over $\Fqqq$, so $\ocb^3=id$.

\subsection{Carrier points}

An $\Fq$-subplane exterior to a line $\ell$ determines two \emph{carrier} points on $\ell$ as follows.
Consider  the collineation group $\bar G=\PGL(3,q^3)$ acting on $\PG(2,q^3)$.
Let $\bar G_{\pi,\ell}$ be the subgroup of $\bar G$
fixing an $\Fq$-subplane $\bar \pi$, and a line $\bar \ell$ exterior to $\bar \pi$. 
Then 
$\bar G_{\pi,\ell}$ is cyclic of order $q^2+q+1$,  acts regularly on the points and on the lines of $\bar \pi$, and  fixes exactly three points of $\PG(2,q^3)$,  called the \emph{$(\bar \pi,\bar \ell)$-carriers} of $\bar \pi$, two of which lie on $\bar \ell$.
That is, the $(\bar \pi,\bar \ell)$-carriers are the three fixed points $$\bar {\ell} \cap
\bar {\ell\,}^{\ocpi},\quad \bar {\ell} \cap
 \bar {\ell\,}^{\ocpis},\quad \bar {\ell\,}^{\ocpi} \cap
\bar {\ell\,}^{\ocpis}.$$

    \subsection{Variety-extensions}
    
  Let $f_1(x_0,\ldots,x_n)=0, \ldots, f_k(x_0,\ldots,x_n)=0$ be $k$  homogeneous $\Fq$-equations (that is, all coefficients of $f_i$ lie in $\Fq$).
These equations give rise to a variety in $\PG(n,q)$, the pointset of the variety is denoted $V(f_1,\ldots,f_k)$ and is the set of all points in $\PG(n,q)$ which satisfy all $k$ equations. In this article we primarily work with quadrics, that is the case where $f_i$ are homogeneous of degree $2$. 
The
 pointset of a variety   
     in $\PG(n,q)$ has a natural extension to the pointset of a variety in the cubic extension $\PG(n,q^3)$ and to $\PG(n,q^6)$ (a sextic extension of $\PG(n,q)$, and a quadratic extension of $\PG(n,q^3)$).  
Let $\K=V(f_1,\ldots,f_k)$, then $\Kstar$ denotes the set of points in $\PG(n,q^3)$ which satisfy the (same) $k$ equations $f_1(x_0,\ldots,x_n)=0, \ldots, f_k(x_0,\ldots,x_n)=0$. Similarly,  $\Kstarstar$ denotes the set of points in $\PG(n,q^6)$ which satisfy the  $k$ equations $f_1(x_0,\ldots,x_n)=0, \ldots, f_k(x_0,\ldots,x_n)=0$.

In particular, if $\Pi_r$ is an $r$-dimensional subspace of $\PG(n,q)$, then $\Pirstar$ is the natural extension to an $r$-dimensional subspace of $\PG(n,q^3)$, and $\Pirstarstar$ is the natural extension to $\PG(n,q^6)$. Note that if $\Sigma_r$ is an $r$-dimensional subspace of $\PG(n,q^3)$ (possibly disjoint from $\PG(n,q)$), then we denote the (quadratic) extension of $\Sigma_r$ to $\PG(n,q^6)$ by $\Sigmarstarstar$. 

In this article we use the $\star$ and $\blackstar$ notations for varieties in the Bruck-Bose and Bose representations, that is, when $n=5,6,8$. We do not use the $\star,\blackstar$ notation in $\PG(2,q^3)$.

 \subsection{Regular 2-spreads}


We will work with the Desarguesian plane $\PG(2,q^3)$, and both the  
Bruck-Bose and Bose representations, so we work with 2-spreads of $\PG(5,q)$ and $\PG(8,q)$ respectively.
A 2-{\em spread} of $\PG(n,q)$, $n=3s+2$, is a set of  planes that partition the points of 
$\PG(n,q)$.
We use the following construction of a regular 2-spread of $\PG(3s+2,q)$, see \cite{CasseOKeefe}.
Embed $\PG(3s+2,q)$ in $\PG(3s+2,q^3)$ and consider the automorphic collineation in  $\PGammaL(3s+3,q^3)$ of order 3 that  fixes $\PG(3s+2,q)$ pointwise, 
$$X=(x_0,\ldots,x_{3s+2}) \longmapsto 
X^{q}=(x_0^q,\ldots,x_{3s+2}^q).$$
 Let $\Pi $ be an $s$-space in $\PG(3s+2,q^3)$ which is disjoint from $\PG(3s+2,q)$,  such that $\Pi,\Pi^{q},\Pi^{q^{2}}$ span $\PG(3s+2,q^3)$. 
For a point $X\in\Pi$, the plane $\langle X,X^{q}, X^{q^2}\rangle$ meets $\PG(3s+2,q)$ in a plane. The planes $\langle X,X^{q}, X^{q^2}\rangle\cap\PG(3s+2,q)$ for $X\in\Pi$ form a 2-spread of $\PG(3s+2,q)$. The $s$-spaces  $\Pi$, $\Pi^{q}$ and $\Pi^{q^2}$ are called 
 the three {\em transversal spaces} of the spread. A \emph{regular 2-spread} is any    2-spread of $\PG(3s+2,q)$  constructed in this way.

\subsection{The Bruck-Bose representation of $\PG(2,q^3)$ in $\PG(6,q)$}
 
 We  use  the linear representation of a finite
translation plane $\P$ of dimension at most three over its kernel,
an idea which was developed independently by
Andr\'{e}~\cite{andr54} and Bruck and Bose
\cite{bruc64,bruc66}. We will use the vector space setting following Bruck and Bose.
Let $\si$ be a hyperplane of $\PG(6,q)$ and let $\S$ be a regular 2-spread
of $\si$. The phrase {\em a subspace of $\PG(6,q)\takeaway\si$} is used to
  mean a subspace of $\PG(6,q)$ that is not contained in $\si$.  We define the following incidence
structure $\IBB$.
The \emph{points} of $\IBB$ are the points of $\PG(6,q)\takeaway\si$ and the planes of the regular 2-spread $\S$.  The \emph{lines} of $\IBB$ are the 3-spaces of $\PG(6,q)\takeaway\si$ that contain
  an element of $\S$ and the line at infinity corresponds to the set of planes of  $\S$. \emph{Incidence} in $\IBB$ is induced by incidence in
  $\PG(6,q)$.
Then the incidence structure $\IBB$ is isomorphic to $PG(2,q^3)$.

Throughout this article, we let $\S$ denote a regular 2-spread of $\si\cong\PG(5,q)$. 
The regular 2-spread $\S$ has three conjugate transversals lines  in $\sistar\cong\PG(5,q^3)$ which we denote by $g,\ g^q,\ g^{q^2}$. The transversal lines of the regular 2-spread $\S$ play an important role in characterising varieties of $\PG(2,q^3)$. 

We also need to consider the quadratic extension of $\sistar$, that is $\sistarstar\cong\PG(5,q^{6})$, and we denote the extension of the three transversal lines $g,\ g^q,\ g^{q^2}$ to lines of $\PG(5,q^{6})$
 by
${\gstar},\ {\gstar}^q,\  {\gstar}^{q^{2}}$ respectively.

 If $\bar \K$ is a set of points in $\PG(2,q^3)$, then we denote the corresponding set of points in the Bruck-Bose representation by $[\K]$.   So if $\bar P\in\PG(2,q^3)\setminus\li$, then $[P]$ is a point in $\PG(6,q)\setminus\si$. If $\bar P\in\li$, then $[P]$ is  plane of the 2-spread $\mathcal S$ in $\si\cong\PG(5,q)$. Further, in the extension to $\PG(5,q^3)$, the plane $\Pstar$ meets the transversal $g$ in a point $P=\Pstar\cap g$. That is, there is a 1-1 correspondence $$\bar P\in\li\quad \longleftrightarrow \quad  P\in g.$$ 

The representations of $\Fq$-sublines, and tangent and secant $\Fq$-subplanes of $\PG(2,q^3)$ in $\PG(6,q)$ were determined in \cite{BJ-FFA,BJ-iff}. We repeatedly use the following two cases.

\begin{result}\Label{BB-Baer-3} \cite{BJ-FFA,BJ-iff}
Consider  the Bruck-Bose representation of $\PG(2,q^3)$ in $\PG(6,q)$. 
\begin{enumerate}
\item An $\Fq$-subline of $\PG(2,q^3)$ that meets $\li$ in a point $\bar T$ corresponds to a line of $\PG(6,q)\setminus\si$ meeting the spread plane $[T]$. 
\item An $\Fq$-subplane $\bar \pi$ of $\PG(2,q^3)$  secant to $\li$ corresponds  to a plane of $\PG(6,q)\setminus\si$ that meets each plane of the  2-regulus of $\S$ corresponding to the  $\Fq$-subline $\bar \pi\cap\li$.
\end{enumerate}
Moreover, the converse of each of these correspondences holds.
\end{result}

\subsection{The exact-at-infinity Bruck-Bose representation}\Label{sec:exact}

As per our usual convention for the Bruck-Bose representation, the correspondences in Result~\ref{BB-Baer-3} are not necessarily exact at infinity.  For example, we can restate Result~\ref{BB-Baer-3}(2) to be exact at infinity as follows. 

\begin{result}\Label{exact-subplane}{\rm 
Let $\bar\pi$ be an $\Fq$-subplane of $\PG(2,q^3)$ secant to $\li$. In  the \emph{exact-at-infinity} Bruck-Bose representation in $\PG(6,q)$, $[\pi]$ consists of the 2-regulus  $\R$ of $\S$ that corresponds to the $\Fq$-subline $\bar b=\bar\pi\cap\li$, \emph{together with} a  
plane $\alpha$  of $\PG(6,q)\setminus\si$ that meets each plane of $\R$ in a point. }
\end{result}

For clarity, we distinguish between this and the usual convention by including   the phrase ``the \emph{exact-at-infinity} Bruck-Bose representation in $\PG(6,q)$''.

\subsection{The Bose representation of $\PG(2,q^3)$ in $\PG(8,q)$}

Bose \cite{Bose} gave a construction to represent the Desarguesian plane $\PG(2,q^2)$ in $\PG(5,q)$ using a regular $1$-spread. 
More generally, we can use the technique of field reduction to generalise this to represent 
 the Desarguesian plane  $\PG(2,q^h)$ using a regular  $(h-1)$-spread in $\PG(3h-1,q)$,  see for example \cite{fieldreduction}.   This idea goes back to Segre \cite{segre} who introduced Desarguesian spreads arising from field reduction. 

 We work with the following  representation of $\PG(2,q^3)$ using a regular  2-spread in $\PG(8,q)$. 
Let $\mathbb S$ be a regular 2-spread in $\PG(8,q)$.  Let $\IB$ be the incidence structure with {\em points} the $q^6+q^3+1$  planes of  ${\mathbb S}$; {\em lines} the  5-spaces of $\PG(8,q)$ that meet ${\mathbb S}$ in $q^3+1$ planes; and {\em incidence} is inclusion. 
The  
5-spaces of $\PG(8,q)$ that meet ${\mathbb S}$ in $q^3+1$ planes form a dual spread $\mathbb H$  (that is, each 7-space of $\PG(8,q)$ contains a unique 5-space in $\mathbb H$). 
  Then $\IB\cong\PG(2,q^3)$, and this representation is called the {\em Bose representation} of $\PG(2,q^3)$ in $\PG(8,q)$. 
 The regular 2-spread $\mathbb S$ has three conjugate transversal planes which we denote throughout this article by $\Pit$, $\Pit^q$, $\Pit^{q^2}$. Note that $\IB\cong\Pit \cong\PG(2,q^3)$.

We use the following notation. 
    A point $\bar X$ in $\PG(2,q^3)$ has  Bose representation  the plane of $\mathbb S$ denoted by $\Bo{X}$. Further, $\bar X$ corresponds to a unique point of $\Pit$ denoted $X$,
 where   $\Bo{X}\star\cap\Pit=X$ and  $\Bo{X}\star=\langle  X, X^q,X^{q^2}\rangle$. 
 More generally, if $\bar\K$ is a set of points of $\PG(2,q^3)$, then $\lbose \K\rbose=\{\Bo{X}\st \bar X\in\bar \K\}$ denotes the corresponding set of planes in the Bose representation in $\PG(8,q)$, and $\K=\{\Bo{X}\star\cap\Pit\st \bar X\in\bar \K\}
 $ denotes the corresponding set of points of $\Pit$.

 So we have the following correspondences:
 $$
\begin{array}{ccccccccccc}
\PG(2,q^3)&\cong&\Pit&\cong&\IB\\
\bar X &\longleftrightarrow&
X
&\longleftrightarrow&
\Bo{X}=\langle  X, X^q,X^{q^2}\rangle\cap\PG(8,q).
\end{array}$$

\subsection{$\Fq$-substructures in the Bose representation}\Label{sec:scroll-plane}

We need to look at $\Fq$-sublines, $\Fq$-subplanes and  $\Fq$-conics of  $\PG(2,q^3)$ in the  Bose representation in $\PG(8,q)$. The Bose representation of $\Fq$-sublines and $\Fq$-subplanes is proved in    \cite{lunardon} and also in \cite[Theorem 2.6]{fieldreduction} using field reduction techniques.
The Bose representation of conics  and $\Fq$-conics of $\PG(2,q^3)$ are determined in \cite{BJW3}.
 Further, \cite{BJW3} looks at these structures in the extension of $\PG(8,q)$  to $\PG(8,q^3)$ and $\PG(8,q^6)$. We briefly summarise the results we need  in order to establish the notation we will use.

In $\PG(2,q^3)$, let $\bar \pi$ be an $\Fq$-subplane and $\bar b$ an $\Fq$-subline of the line $\bar \ell_b$. Define  
  ${\ocpi}$ and $\ocb$ as in (\ref{defcpi}) and (\ref{defcb}). Denote the corresponding maps which act on the points of 
   the transversal plane $\Pit$ and the line $\ell_b\subset\Pit$ 
  by ${\cpi}$ and $\cb$ respectively. Note that ${\cpi}$ and $\cb$  are \emph{not} collineations of $\PG(8,q^3)$, they only act on the points of   $\Pit$ and $\ell_b$ respectively. Further $\cpi$ and $\cb$ induce collineations acting on the points of $\Pit\blackstar$ and $\ellbstarstar$ respectively. 
  
In $\PG(8,q^6)$, we define the \emph{$\pi$-scroll-plane} of a point $X$ in $\Pit\blackstar$ to be  the plane  $$\lround X\rround_\pi= \big\langle\, X, (X^{{\mathsf c}^5_\pi})^q, (X^{\mathsf c^4_\pi})^{q^2}\,\big\rangle.$$
Note that if $X\in\pi$, then $\lround X\rround_\pi=\Bo{X}\blackstar$. Further, if $X\in\Pit\setminus\pi$, then 
$\lround X\rround_\pi$ simplifies to $\lround X\rround_\pi= \big\langle\, X, (X^{{\mathsf c}^2_\pi})^q, (X^{\cpi})^{q^2}\,\big\rangle$, and is disjoint from $\PG(8,q)$. 
The \emph{$b$-scroll-plane} of a point $X\in \ellbstarstar$ in $\PG(8,q^6)$ is the plane
$\lround X\rround_b= \big\langle\, X, (X^{{\mathsf c}^5_b})^q, (X^{\mathsf c^4_b})^{q^2}\,\big\rangle$ which lies in the 5-space $\langle \ell_b, \ell_b^q, \ell_b^{q^2}\rangle\blackstar$. 
 We quote the $\Fq$-conic result which we will need here. 
 
%

\begin{result}\cite{BJW3}\Label{conic-subplane-B}
    Let $\bar \C$ be an $\Fq$-conic in the $\Fq$-subplane $\bar\pi$ of $\PG(2,q^3)$. 
In $\PG(8,q)$, the planes of  $\Bo{\C}$ form a variety $\VC=\V^6_3$ which is the intersection of nine quadrics: $\VC=\Q_1\cap\cdots\cap\Q_9$. 
In $\PG(8,q^3)$, the points of the variety  $\VCstar=\Qonestar\cap\cdots\cap\Qninestar$ coincide with the points on the planes $  \{\lround X\rround_\pi\st X\in\Cplus\}$; and in $\PG(8,q^6)$,  the  points of the variety  $\VCstarstar=\Qonestarstar\cap\cdots\cap\Qninestarstar$  coincide with the points on the planes $ \{\lround X\rround_\pi\st X\in\Cplusplus\}$.
\end{result}

 \subsection{The Bruck-Bose representation inside the Bose representation}

 We can construct the Bruck-Bose representation of $\PG(2,q^3)$ by intersecting a 6-space with the Bose representation of $\PG(2,q^3)$ in $\PG(8,q)$ as follows. 
 Let $\Sigma_{6,q}$ be a 6-space of $\PG(8,q)$ that  contains a unique 5-space of the dual spread  $\mathbb H$, we denote this 5-space by  $\si$. The extension of $\Sigma_{6,q}$ to $\Sigma_{6,q}^\star$ meets the transversal plane $\Pit $ in a line $ g $. Further, $\langle  g , g ^q,g^{q^2}\rangle\cap\PG(8,q)=\si\cong\PG(5,q)$. The intersection of the Bose representation with $\Sigma_{6,q}$ gives the Bruck-Bose representation of $\PG(2,q^3)$ in the 6-space $\Sigma_{6,q}$. That is, $\IBB=\IB \cap\Sigma_{6,q}$. In particular the 2-spread $\S$ of the Bruck-Bose representation is contained in the 2-spread $\mathbb S$ of the Bose setting, that is $\S=\mathbb S\cap\si$. 
 
 So we have the following correspondences:
 $$
\begin{array}{ccccccccccc}
\PG(2,q^3)&\cong&\Pit&\cong&\IB&\cong&\IBB\\
\bar P &\longleftrightarrow&
P
&\longleftrightarrow&
\Bo{P}=\langle\, P,\ P^q,\ P^{q^2}\,\rangle\cap\PG(8,q)
&\longleftrightarrow&
[P]=\Bo{P}\cap\Sigma_{6,q}.
\end{array}$$

Note that when considering  the Bruck-Bose representation as $\IBB=\IB\cap\Sigma_{6,q}$,  we obtain a representation which is exact on $\si$, that is, we have the ``exact-at-infinity Bruck-Bose representation in $\PG(6,q)$'' described in Section~\ref{sec:exact}.

\subsection{Notation summary}\Label{sec-notn}

This article works with $\Fq$-conics in $\PG(2,q^3)$ in the planar setting, the $\PG(6,q)$ Bruck-Bose setting and the $\PG(8,q)$ Bose setting. We have designed the notation to help distinguish between these settings, and a summary of the notation used is given here.
 
 \begin{itemize}
 
  \item In $\PG(2,q^3)$
   \begin{itemize}
 \item[$\cdot$]  Objects in $\PG(2,q^3)$ are indicated with an overline $\bar{\phantom{x}}$.
 \item[$\cdot$] $\bar \C$ denotes an $\Fq$-conic in $\PG(2,q^3)$, and $\bar \Cplus$ denotes the unique $\Fqqq$-conic  in $\PG(2,q^3)$ containing $\bar \C$.
  \item[$\cdot$] For an $\Fq$-subplane $\bar \pi$ contained in $\PG(2,q^3)$, $\ocpi:\bar X\mapsto B \bar X^q$  (defined in (\ref{defcpi})) generates the unique collineation subgroup of order 3 acting on $\PG(2,q^3)$ which fixes $\bar \pi$ pointwise.
 \item[$\cdot$] For an $\Fq$-subline $\bar b$ contained in a line $\bar \ell_b$ of $\PG(2,q^3)$, ${\bar {\mathsf c}}_b:\bar X\mapsto D \bar X^q$ (defined in (\ref{defcb}))  generates the unique collineation subgroup of order 3 acting on $\bar \ell_b$, which fixes $\bar b$ pointwise.

\item[$\cdot$] We can extend $\PG(2,q^3)$ to $\PG(2,q^6)$, and let $\liplusplus$ denote the quadratic extension of $\li$, and $\barCplusplus$ denote the unique $\Fqqqqqq$-conic  containing $\bar \Cplus$.

  \end{itemize}
 
 \item In $\PG(n,q)$, $n>2$
   \begin{itemize}

 \item[$\cdot$] $\N_r$ denotes an $r$-dim nrc.
  \item[$\cdot$] If $\Q$ is a quadric in  $\PG(n,q)$,  denote the extension to $\PG(n,q^3)$ by $\Qstar$, and the extension to $\PG(n,q^6)$ by $\Q\blackstar$. 
   \item[$\cdot$] If $\N$ is a normal rational curve in  $\PG(n,q)$,  denote the extension to $\PG(n,q^3)$ by $\Nstar$, and the extension to $\PG(n,q^6)$ by $\N\blackstar$. 
  \item[$\cdot$] If $X=(x_0,\ldots,x_n)$, then $X^q=(x_0^q,\ldots,x_n^q)$. 
 \item[$\cdot$] $\mathsf e$ denotes the conjugate map associated with the square extension from $\PG(n,r)$ to $\PG(n,r^2)$, that is, $${\mathsf e}\colon X=(x_0,\ldots,x_n)\mapsto X^{\mathsf e}=(x_0^r,\ldots,x_n^r).$$
 \end{itemize}
 \item In $\IB$
 \begin{itemize}
 \item[$\cdot$]   $\mathbb S$ is a regular 2-spread in $\PG(8,q)$.
  \item[$\cdot$]   $\mathbb S$ has three transversal planes in $\PG(8,q^3)$, denoted $\Pit$, $\Pit^q$, $\Pit^{q^2}$. 
 \item[$\cdot$] 
 A point $\bar P$ in $\PG(2,q^3)$ corresponds to a point $P$ in the transversal plane $\Pit$, and to a plane $\Bo{P}$ of $\mathbb S$, where $\Bo{P}=\langle\, P,\, P^q,\, P^{q^2}\, \rangle\cap\PG(8,q)$.

 \end{itemize}
  \item In $\IBB$
 \begin{itemize}
\item[$\cdot$]   $\S$ is a regular 2-spread in the 5-space at infinity $\si\cong\PG(5,q)$. 
\item[$\cdot$]  $\S$ has transversal lines denoted 
 $g,g^{q},g^{q^2}$ which lie in $\sistar\setminus\si$.
\item[$\cdot$]   For a point $\bar X\in\PG(2,q^3)\setminus \li$, we denote corresponding point of $\PG(6,q)\setminus\si$ by  $[X]$.
\item[$\cdot$]   If $\bar X\in\li$, then we denote the corresponding point on $g$ by $X$ and the corresponding spread plane by $[X]$, so $[X]=\langle \, X, \,X^q, \,X^{q^2}\, \rangle \cap\si$ and $X=[X]\star\cap g$. 
  \end{itemize}
 \item  in $\IBB=\IB\cap\Sigma_{6,q}$
\begin{itemize}
\item[$\cdot$]  The 6-space $\Sigmaqstar$ meets the transversal plane $\Pit$ of $\mathbb S$ in the transversal line $g$ of $\S$.
\item[$\cdot$]  For a point $\bar X\in\PG(2,q^3)$, we have $[X]=\Bo{X}\cap\Sigma_{6,q}$.
\end{itemize}
\end{itemize}

 \section{Preliminaries}\Label{sec:Fqconic}

\subsection{$\Fq$-conics}

We define an $\Fqqq$-conic in $\PG(2,q^3)$ to be a non-degenerate conic of $\PG(2,q^3)$.  We define an $\Fq$-conic of $\PG(2,q^3)$ to be a non-degenerate conic of an $\Fq$-subplane of $\PG(2,q^3)$. That is, an $\Fq$-conic is projectively equivalent to a set of points in $\PG(2,q)$ that satisfy a non-degenerate homogeneous quadratic equation over $\Fq$.  

Let $\bar \C$ be an $\Fq$-conic in the $\Fq$-subplane $\bar \pi$,  let $\bar \Cplus$ be the unique $\Fqqq$-conic of $\PG(2,q^3)$ containing $\bar \C$, and let $\barCplusplus$ denote the 
unique $\Fqqqqqq$-conic in the quadratic extension $\PG(2,q^6)$ that contains $\bar\Cplus$. 
Let $\ocpi$ be defined as in (\ref{defcpi}), the following useful property is straightforward to verify. 

\begin{result}\Label{result-conj}
In $\PG(2,q^3)$, let $\bar \C$ be an $\Fq$-conic. If $\bar X\in \bar \Cplus$, then $\bar X^{\ocpi}, \bar X^{\ocpis}\in\bar \Cplus$. 
Further, in the quadratic extension $\PG(2,q^6)$, 
if $\bar Y\in\barCplusplus$, then $\bar Y^{\ocpii}\in\barCplusplus$, $i=1,\ldots,5$. 
\end{result}

The Bruck-Bose representation of $\PG(2,q^3)$ is constructed using a line at infinity $\li$. In order to study $\Fq$-conics in the Bruck-Bose representation, we need to look at the position of an $\Fq$-conic in relation to $\li$. 
There are a total of eleven different settings for  an $\Fq$-conic of $\PG(2,q^3)$ in relation to $\li$. These are described in the next result,  the proof of which is straightforward and is omitted.

\begin{result}\Label{Fqconic-cases} In $\PG(2,q^3)$, let $\li$ denote the line at infinity. Let $\bar \pi$ be an $\Fq$-subplane and let $\ocpi$ be the   
  collineation which fixes $\bar \pi$ pointwise as defined  in (\ref{defcpi}). 
 Let $\bar \C$ be an $\Fq$-conic in  $\bar\pi$,  let $\bar \Cplus$ be the unique $\Fqqq$-conic of $\PG(2,q^3)$ containing $\bar \C$, and in the extension to $\PG(2,q^6)$, let $\barCplusplus$ be the
unique $\Fqqqqqq$-conic of $\PG(2,q^6)$ that contains $\bar\Cplus$. Let $\liplusplus$ denote the extension of $\li$ to $\PG(2,q^6)$. 
 \begin{itemize}
  
\item Suppose {\sl $\bar \pi$ is secant to $\li$}, then either
\begin{enumerate}


\item[\SI] {\sl $\bar\Cplus$ is secant to $\li$:\ } $\bar \C\cap \li=\{\bar P,\bar Q\}$, $\bar P\neq\bar Q$,  $\bar P,\bar Q\in \bar \pi$.
\item[\SII] {\sl $\bar\Cplus$ is tangent to $\li$:\ } $\bar \C\cap \li=\{\bar P\}$,  $\bar P\in \bar \pi$.

\item[\SIII]  {\sl $\bar\Cplus$ is exterior to $\li$:\ }  $\bar \C\cap \li=\emptyset$, in which case $\bar \Cplus\cap \li=\emptyset$, and  $\barCplusplus\cap \liplusplus=\{\bar P,\bar Q\}\subset\liplusplus\setminus\li$, $\bar P\neq \bar Q$. Moreover,    $\bar Q=\bar P^{\ocpi}=\bar P^{q^3}$, $\bar P^{\ocpis}=\bar P$, and if $\bar b=\bar \pi\cap\li$, then $\bar P^{\ocb}=\bar P^{\ocpi}=\bar Q$. 

\end{enumerate}

\item Suppose {\sl $\bar \pi$ is tangent to $\li$}, with $\bar T=\bar \pi\cap \li$. Then either
\begin{enumerate}
\item[\TI]  {\sl $\bar\Cplus$ is secant to $\li$:\ } $\bar T\in\bar \C$, in which case $\bar \Cplus\cap \li=\{\bar T,\bar Q\}$ with $\bar T\neq \bar Q$. Note that  $\bar T=\bar T^{\ocpi}=\bar T^{\ocpi^2}$ and $\bar Q^{\ocpi},\bar Q^{\ocpi^2}\notin \li$.

\item[\TII]  {\sl $\bar\Cplus$ is secant to $\li$:\ } $\bar T\notin \bar \C$ and $\bar \Cplus\cap \li=\{\bar P,\bar Q\}$, $\bar P\neq\bar Q$. So $\bar P,\bar Q$ both have orbit size 3 under $\ocpi$, and $\bar P^{\ocpi},\bar P^{\ocpis},\bar Q^{\ocpi},\bar Q^{\ocpis}\notin \li$.

\item[\TIII]  {\sl $\bar\Cplus$ is tangent to $\li$:\ } $\bar T\notin \bar \C$ and $\bar \Cplus\cap \li=\{\bar P\}$. If $q$ is even, then $\bar T$ is the nucleus of $\bar \C$. The point $\bar P$ has orbit size 3 under $\ocpi$, and $\bar P^{\ocpi},\bar P^{\ocpis}\notin \li$.

\item[\TIV]  {\sl $\bar\Cplus$ is exterior to $\li$:\ } $\bar T\notin \bar \C$ and $\bar \Cplus\cap \li=\emptyset$, then $\barCplusplus\cap \liplusplus=\{\bar P,\bar Q\}\subset\liplusplus\setminus\li$, $\bar P\neq \bar Q$. In this case $\bar Q=\bar P^{q^3}=\bar P^{\ocpic}$, $\bar P$ has orbit size 6 under $\ocpi$, and $\bar P,\bar P^{\ocpic}\in \li$, 
$\bar P^{\ocpi},\bar P^{\ocpis},\bar P^{\ocpif},\bar P^{\ocpiff}\notin \li$. 
\end{enumerate}

\item Suppose {\sl $\bar \pi$ is exterior to $\li$}, denote the $(\bar \pi,\li)$-carriers which lie on $\li$ by   $\bar E,\bar E^{\ocpi}$. Then either
\begin{enumerate}
\item[\EI] {\sl $\bar\Cplus$ is secant to $\li$:\ }  $\bar \Cplus\cap \li=\{\bar E,\bar E^{\ocpi}\}$.

\item[\EII] {\sl $\bar\Cplus$ is secant to $\li$:\ } $\bar \Cplus\cap \li=\{\bar P,\bar Q\}$, with $\bar P\neq \bar Q$ and  $\{\bar P,\bar Q\}\cap\{ \bar E,\bar E^{\ocpi}\}=\emptyset$. In this case  $\bar P,\bar Q$ both have orbit size 3 under $\ocpi$, and $\bar P^{\ocpi},\bar P^{\ocpis},\bar Q^{\ocpi},\bar Q^{\ocpis}\notin \li$.

\item [\EIII] {\sl $\bar\Cplus$ is tangent to $\li$:\ }  $\bar \Cplus\cap\li=\{\bar P\}$. This case only occurs if $q$ is odd, in which case $\bar P\cap \{ \bar E,\bar E^{\ocpi}\}=\emptyset$ and  $\bar P$ has orbit size 3 under $\ocpi$, and $\bar P^{\ocpi},\bar P^{\ocpis}\notin \li$.

\item[\EIV] {\sl $\bar\Cplus$ is exterior to $\li$:\ } $\bar \Cplus\cap \li=\emptyset$, 
in which case $\barCplusplus\cap \liplusplus=\{\bar P,\bar Q\}\subset\liplusplus\setminus\li$, $\bar P\neq \bar Q$. Further  $\bar Q=\bar P^{q^3}=\bar P^{\ocpic}$, $\bar P$ has orbit size 6 under $\ocpi$, and $\bar P,\bar P^{\ocpic}\in \li$, 
$\bar P^{\ocpi},\bar P^{\ocpis},\bar P^{\ocpif},\bar P^{\ocpiff}\notin \li$.

\end{enumerate}

\end{itemize}
\end{result}

\emph{Remark} There 
 are $q^2+q+1$\, $\Fq$-conics of the type described in \EI, forming a circumscribed bundle of conics of $\overline\pi$.  
There are $q^2+q+1\,$ $\Fq$-conics of the type described in \EIII, forming an inscribed bundle of conics of $\overline\pi$.

%
%
%
%
%

\subsection{T-planes}

We begin with a lemma describing  how a $\pi$-scroll-plane (defined in Section~\ref{sec:scroll-plane}) meets a 5-space. 
In $\PG(8,q^3)$, let $\Pit,\Pit^q,\Pit^{q^2}$ be the transversal planes of the Bose spread $\mathbb S$. 
We use the following terminology in $\PG(8,q^3)$. 
\begin{itemize}
\item A \emph{T-point}  is a point which lies in one of the  transversal planes $\Pit,\Pit^q,\Pit^{q^2}$.
\item A \emph{T-line}  is a line which meets two of  the  transversal planes.
\item A \emph{T-plane} is a plane that meets all three of  the  transversal planes.  
\end{itemize}
Similarly, we can define T-points,  T-lines and T-planes in the quadratic extension $\PG(8,q^6)$ in terms of the extended transversal planes.

\begin{lemma}\Label{cpiplus} 
In the Bose setting, let $\Pit$ be a transversal plane of the Bose spread $\mathbb S$. Let $g$ be a line of $\Pit$,  let $\pi$ be an $\Fq$-subplane of  $\Pit$, and consider the 5-space  $\Pi_g=\langle g,g^q,g^{q^2}\rangle\cap\PG(8,q)$.
\begin{enumerate}
\item In $\PG(8,q^3)$, for a point $X\in\Pit$, $\lround X\rround_\pi\cap\Pigstar$ 
 is either
 $\emptyset$,
 a T-point, 
 a T-line, or
 a T-plane.
\item In $\PG(8,q^6)$, for a point $X\in\Pit\blackstar$,  $\lround X\rround_\pi\cap\Pigstarstar$ is either $\emptyset$, a T-point, a T-line, or a T-plane.
\end{enumerate}
\end{lemma}

\begin{proof}
Note that the lines $g,g^q,g^{q^2}$ are transversal lines of a regular 2-spread $\S$ of   $\Pi_g$, where $\S=\mathbb S\cap\Pi_g$. In $\Pigstar$: a T-point is a point of one of the transversal lines $g,g^q,g^{q^2}$; a T-line is 
 a line which meets two of $g,g^q,g^{q^2}$; and a T-plane is a plane that meets all three of $g,g^q,g^{q^2}$. Similar to the proof of  \cite[Cor 2.2]{BJW3}, we can show that two distinct  T-planes in $\Pigstar$ meet in either $\emptyset$, a T-point, or a T-line. 

%
Suppose there exists a $\pi$-scroll-plane where $\lround X\rround_\pi\cap\Pigstar$ is not one of: $\emptyset$, a T-point, a T-line or a T-plane. Then either  there is a point $L\in\lround X\rround_\pi\cap\Pigstar$ with $L$   on a T-line and  not a T-point; or there is a point $K\in\lround X\rround_\pi\cap\Pigstar$ with $K$  not on a T-line.
First suppose that $\lround X\rround_\pi$ is not contained in $\Pigstar$ and that  there is a point $L\in\lround X\rround_\pi\cap\Pigstar$ with $L$   on a T-line $m$, $L$ not a T-point. Let $\beta$ be one of the  T-planes of $\Pigstar$ containing  the line $m$. Then $\beta$ is a T-plane contained in $\Pigstar$, so  $\beta\neq \lround X\rround_\pi$ and $L\in\beta\cap \lround X\rround_\pi$. This contradicts  our statement that  that two distinct  T-planes in $\Pigstar$ meet in either $\emptyset$, a T-point, or a T-line. 
Now suppose 
 that  there is a point $K\in\lround X\rround_\pi\cap\Pigstar$ with $K$  not on a T-line.  Then as in the previous case, $K$ lies on a unique T-plane $\alpha\subset\Pigstar$. Note that $\alpha\neq \lround X\rround_\pi$ and $\alpha$ is also a T-plane. That is,  $K$ lies on two distinct T-planes, a contradiction. This completes the proof of part 1, the proof of part 2 is similar. 
\end{proof}

\subsection{A first description of $[\C]$}

We will determine the representation of an $\Fq$-conic of $\PG(2,q^3)$ in the \emph{exact-at-infinity}  Bruck-Bose representation  in $\PG(6,q)$. 
Fundamental to determining this structure is the interplay between the Bose representation $\IB$ in $\PG(8,q)$ and  the Bruck-Bose representation $\IBB$ in $\PG(6,q)$ considered as  $\IBB=\IB\cap\Sigma_{6,q}$.   Using Result~\ref{conic-subplane-B}  we first prove the following coarse description in the exact-at-infinity Bruck-Bose representation.
Recall that an \emph{affine point} of $\PG(6,q)$ means a point in $\PG(6,q)\setminus\si$.

\begin{theorem}\Label{fqconic-order6}
Let $\bar\C$ be an $\Fq$-conic in $\PG(2,q^3)$. In  the exact-at-infinity Bruck-Bose  representation in $\PG(6,q)$, the pointset of $[\C]$ forms a variety $\VCBB$ which is the intersection of nine quadrics. Further, $\VCBB$ is  generically a curve of degree $6$. 
\end{theorem}

\begin{proof} We look at the $\PG(6,q)$ Bruck-Bose representation of $\bar \C$ using the $\PG(8,q)$ Bose representation. 
Let $\bar \C$ be an $\Fq$-conic in  $\PG(2,q^3)$. In the Bose representation in  $\PG(8,q)$, $\Bo{\C}$ is a set of $q+1$ planes. By Result~\ref{conic-subplane-B}, the  pointset of  $\Bo{\C}$ forms a variety $\VC=\V^6_3$ of dimension 3 and degree 6 which is the intersection of nine quadrics $\VC=\Q_1\cap\cdots\cap\Q_9$.  
Let $\Sigma_{6,q}$ be a 6-space whose extension to $\PG(8,q^3)$ meets the transversal plane $\Pit$ in the line $g$.  So in the Bruck-Bose  setting $\IBB=\IB\cap\Sigma_{6,q}$, we have 
hyperplane at infinity $\si=\langle g,g^q,g^{q^2}\rangle\cap\PG(8,q)$. 
The $\Fq$-conic $\bar\C$ corresponds to a set of points $[\C]=\Bo{\C}\cap\Sigma_{6,q}$ in this Bruck-Bose setting. These points form a variety $\VCBB=\VC\cap\Sigma_{6,q}=\Q_1\cap\cdots\cap\Q_9\cap\Sigma_{6,q}$. 
The 6-space $\Sigma_{6,q}$ is a variety $\V^1_6$. Generically, we have $\V^6_3\cap\V^1_6=\V^6_1$, so  the variety $\VC=\V^6_3$ generically meets   $\Sigma_{6,q}$   in a curve of degree 6, that is, $\VCBB=\V^6_1$.
\end{proof}

\subsection{Transversals}

The next result   determines how the exact-at-infinity  Bruck-Bose representation of an $\Fq$-conic $\bar\C$ meets the transversal lines $g,g^q,g^{q^2}$ of the regular 2-spread $\S$.  Note that the variety   $\V([\C])\star$ 
 may well contain planes or lines in $\sistar$, and these may account for the intersections with  $g$ described in the next theorem. 

\begin{theorem}\Label{cplus-g}
Let $\S$ be a regular 2-spread in the hyperplane at infinity $\si$ of $\PG(6,q)$, with transversal lines denoted $g,g^q,g^{q^2}$.
 Let $\bar \C$ be an $\Fq$-conic  of $\PG(2,q^3)$, then in  the exact-at-infinity Bruck-Bose representation in $\PG(6,q)$,  $\V([\C])\star\cap g=\Cplus\cap g$.
  Hence $\bar P\in\bar\Cplus\cap\li$ if and only if $P\in \Cplus\cap g$. 
  \end{theorem}

\begin{proof}  
Let $\bar \C$ be an $\Fq$-conic  of $\PG(2,q^3)$, in the Bose representation we have from Result~\ref{conic-subplane-B} that 
\begin{eqnarray}\label{eqn1}
\VCstar\cap\Pit=\Cplus.
\end{eqnarray}

We interpret this  in the Bruck-Bose representation. 
Let $g$ be a line of $\Pit$, and $\Sigma_{6,q}$ a 6-space whose extension contains the 5-space  $\langle g,g^q,g^{q^2}\rangle$.  
By the proof of Theorem~\ref{fqconic-order6}, 
 $\V([\C])\star= \VC \star \cap \Sigmaqstar$, and intersecting with $g$ yields  $\V([\C])\star\cap g= \VC \star \cap g$. 
Intersecting both sides of (\ref{eqn1})  with $g$, and equating  gives $\V([\C])\star\cap g=\Cplus\cap g$.
\end{proof}

The next result  looks how
  $\pi$-scroll-planes of $\VC$ meet the    5-space $\langle g,g^q,g^{q^2}\rangle$.

\begin{lemma}\Label{scroll-plane-C}
Let $\bar \C$ be an $\Fq$-conic in an $\Fq$-subplane $\bar \pi$ of $\PG(2,q^3)$. As defined in (\ref{defcpi}), let $\ocpi:\bar X\mapsto B\bar X^q$.  In the Bose representation $\PG(8,q)$, let  $\VC$ be the variety of $\PG(8,q)$ whose pointset corresponds to the pointset of $\Bo\C$. Let $g$ be a line of $\Pit$ and consider the 5-space $\Pi_g=\langle g,g^q,g^{q^2}\rangle\cap\PG(8,q)$.
\begin{enumerate}
\item In $\PG(8,q^3)$, 
$\VCstar\cap\Pigstar=\ \{\lround X\rround_\pi \cap \Pigstar\st X^{\mathsf c^i_\pi} \in \Cplus\cap g, \textup{ for some } i\in\{0,1,2\}\}$.

\item  In $\PG(8,q^6)$,  
$\VCstarstar\cap\Pigstarstar=\{\lround X\rround_\pi \cap \Pigstarstar\st X^{\mathsf c^i_\pi} \in \Cplusplus\cap\gstar,  \textup{ for some } i\in\{0,\ldots,5\}\}$.  
\end{enumerate}
\end{lemma}

\begin{proof} 
First consider the extension to $\PG(8,q^3)$. By Result~\ref{conic-subplane-B}, 
    the points of the variety  $\VCstar $ coincide with the points on the planes $  \{\lround X\rround_\pi\st X\in\Cplus\}$, so 
 \begin{eqnarray*}
 \VCstar\cap\Pigstar
 & =&
 \{\lround
X\rround_\pi\cap\Pigstar\st X\in\Cplus\}. 
\end{eqnarray*}
Recall that $\lround X\rround_\pi=\langle X,(X^{\cpi^2})^q, (X^{\cpi})^{q^2}\rangle,$
so we are looking for points $X\in\Cplus$ for which $\lround X\rround_\pi\cap\Pigstar\neq\emptyset$.
By Lemma~\ref{cpiplus}, $\lround
X\rround_\pi\cap\Pigstar$ is either 
\begin{itemize}
\item $\emptyset$, or
\item  $X$ or $(X^{\cpi^2})^q$ or $ (X^{\cpi})^{q^2}$, or 
\item a line joining two of $X,(X^{\cpi^2})^q, (X^{\cpi})^{q^2}$, or 
\item  $\lround X\rround_\pi$. 
\end{itemize}
Hence we want points $X\in\Cplus$ for which at least one of the points $X,(X^{\cpi^2})^q, (X^{\cpi})^{q^2}$ lies in $\Pigstar$.

As $\Cplus$ lies in the transversal plane $\Pit$ and $\Pigstar$ meets $\Pit$ in the line $g$, a point $X\in\Cplus$ lies in 
$\Pigstar$ if and only if  $X\in\Cplus\cap g$. For a point 
$X^{\cpi^2}\in\Cplus$, 
the point 
$(X^{\cpi^2})^q$ lies in the transversal plane $\Pit^{q}$. Further,  $\Pigstar$ meets $\Pit^{q}$ in the line $g^{q}$. 
Hence $(X^{\cpi^2})^q
\in(\Cplus)^q\cap\Pigstar$  if and only if 
 $X^{\cpi^2}\in\Cplus\cap g$.
Similarly, for a point $X^{\cpi}\in\Cplus$, we have $(X^{\cpi})^{q^2}\in(\Cplus)^{q^2}\cap\Pigstar$
  if and only if $X^{\cpi}\in\Cplus\cap g$. 
That is, for $X\in\Cplus$, $\lround X\rround_\pi\cap\Pigstar\neq\emptyset$ if and only if at least one of $X,X^{\cpi},X^{\cpis}$ lies in $\Cplus\cap g$, proving part 1.
The case 
 in 
$\PG(8,q^6)$ is similar. 
\end{proof}

\section{$\Fq$-conics  in  the exact-at-infinity Bruck-Bose representation}\Label{sec4}

Let $\bar \C$ be an $\Fq$-conic in an $\Fq$-plane $\bar \pi$ of $\PG(2,q^3)$. In this section 
we determine in more detail the structure of the  variety $\VCBB$ in the exact-at-infinity Bruck-Bose representation. We look at  the three cases where $\bar\pi$ is  secant, tangent and exterior to $\li$ separately.

\subsection{$\Fq$-conics in an $\Fq$-subplane secant to $\li$}\Label{sec-secant}

We begin by looking at $\Fq$-conics contained in an $\Fq$-subplane  that is secant to $\li$. 

Let $\bar\pi$ be an $\Fq$-subplane of $\PG(2,q^3)$ that is secant to $\li$. So $\bar b=\bar \pi\cap\li$ is an $\Fq$-subline. Recall the $b$-scroll-plane of a point 
was defined  in the Bose representation in Section~\ref{sec:scroll-plane}. 
As $\bar b$ is an $\Fq$-subline of $\li$, we can similarly define the $b$-scroll-plane  in the Bruck-Bose representation. In the Bruck-Bose setting,  the $b$-scroll-plane of a point $X\in g$ (or   $\gstar$)  is the plane     $\lround X\rround_b=\langle X, ({X}^{\mathsf c^5_b})^q, ({X}^{\mathsf c^4_b})^{q^2}\,\rangle$, which lies in the 5-space $\sistar$ (or $\sistarstar$). 
 This simplifies      to   $\lround X\rround_b=\Bo{X}\star$ for $X\in b$; and  $\lround X\rround_b=\langle X, ({X}^{\mathsf c^2_b})^q, ({X}^{\cb})^{q^2}\,\rangle$ for $X\in g\setminus b$.

\begin{theorem}\Label{conic-sec}
Let $\bar \C$ be an $\Fq$-conic in a secant $\Fq$-subplane $\bar \pi$ of $\PG(2,q^3)$.
In  the exact-at-infinity Bruck-Bose  representation in $\PG(6,q)$, the  curve $\VCBB$ decomposes into a non-degenerate conic $\N_2$  (which lies in a plane of $\PG(6,q)\setminus\si$ that meets $q+1$ spread planes), together with  two planes in $\si$ (possibly repeated, possibly in the extension $\sistarstar\setminus\si$).
\end{theorem}

\begin{proof} Let $\bar \C$ be an $\Fq$-conic in a secant $\Fq$-subplane $\bar \pi$ of $\PG(2,q^3)$.
By  Result~\ref{exact-subplane}, in the exact-at-infinity Bruck-Bose representation in $\PG(6,q)$: 
 the affine points of $[\pi]$ are the affine points of a plane $\alpha_\pi$; and the points of $[\pi]$ in $\si$ are the points of 
the $2$-regulus of $\S$ which $\alpha_\pi$ meets. 
Moreover   $\bar \pi$ and $\alpha_\pi$ are in 1-1 correspondence. So corresponding to $\bar \C$ is a non-degenerate conic $\N_2$ in $\alpha_\pi$.  
So the variety $\VCBB$ contains the conic $\N_2$. 
We show that   the variety $\VCBB$  reduces into $\N_2$ and  two planes contained in $\si$ (or an extension).
There are three cases to consider, depending on how $\bar\C$  meets $\li$. We label the cases to be consistent with Result~\ref{Fqconic-cases}.

Cases \SI and \SII. Suppose $\bar\C\cap\li=\{\bar P,\bar Q\}$, possibly $\bar P=\bar Q$.
In  the exact-at-infinity Bruck-Bose  setting, 
 $[\C]$ contains 
 the spread planes 
 $[P]$, $[Q]$. Moreover, in this case $\N_2$ 
  meets $\si$  in the real  points $[P]\cap\alpha_\pi$ and $[Q]\cap\alpha_\pi$ (possibly repeated).

Case \SIII. Suppose $\bar\C\cap\li=\emptyset$, so by Result~\ref{Fqconic-cases},  $\barCplusplus\cap \liplusplus=\{\bar P,\bar Q\}\subset\liplusplus\setminus\li$, $\bar P\neq\bar Q$. 
   To determine the  component of $\VCBB$ at infinity, we work in the Bose representation. In $\PG(8,q^3)$: $\li$ corresponds to a line of $\Pit$ denoted by $g$; $\pi$ is an $\Fq$-subplane of $\Pit$ that meets $g$ in an $\Fq$-subline; $\C$ is an $\Fq$-conic in  $\pi$; $\Cplus\cap g=\emptyset$; and  $\Cplusplus\cap\gstar=\{P,Q\}\subset\gstar\setminus g$.
Let $b=\pi\cap g$, and let $\cpi$, $\cb$  be as defined in 
   (\ref{defcpi}) and (\ref{defcb}), then by Result~\ref{Fqconic-cases}, $Q=P^{\cpi}=P^{\cb}$ and $P^{\cb^2}=P$.
    
    We now look at the  Bruck-Bose setting,   let $\Sigma_{6,q}$ be  a 6-space of $\PG(8,q)$  that contains the 5-space $\si=\langle g,g^q,g^{q^2}\rangle\cap\PG(8,q)$, and use the setting  $\IBB=\IB\cap\Sigma_{6,q}$. 
 We wish to find $\VCBBstarstar\cap\sistarstar$. By Lemma~\ref{scroll-plane-C}, this intersection is contained in  the  $\pi$-scroll-planes 
 $\lround X\rround_\pi$ of $\PG(8,q^6)$ for which $X^{\cpi^i}\in \Cplusplus\cap\gstar$.
 That is, the intersection is contained in the two 
  $\pi$-scroll-planes  
 $\lround P\rround_\pi$ and $\lround Q\rround_\pi$. We look at how these two $\pi$-scroll-planes 
 meet $\sistarstar$.
 As $P,P^{\cb}\in \gstar$ and  $P^{\mathsf c^2_b}=P$, we have $\lround P\rround_\pi=\lround P\rround_b$ and $\lround Q\rround _\pi=\lround Q\rround_b$.
So $\lround P\rround_b=\langle\, P, (P^{\mathsf c^5_b})^q, (P^{\cb^4})^{q^2}\,\big\rangle
=\langle P,Q^q,P^{q^2}\rangle$, which   lies in the 5-space  $\sistarstar=\langle \gstar,\gqstar,\gqqstar\rangle$. Hence 
$\lround P\rround_b$, and similarly $\lround Q\rround_b$, lies in $\VCBBstarstar\cap
\sistarstar$. 
That is, the variety  $\VCBB$ decomposes into $\N_2$ and two  planes at infinity, namely $\lround P\rround_b,\ \lround Q\rround_b$.
\end{proof}
 
We give a complete description of the Bruck-Bose representation of an $\Fq$-conic in a secant $\Fq$-subplane and how the variety meets $\si$. 

\begin{corollary}\Label{cor:sec}
 Let $\bar \C$ be an $\Fq$-conic in a secant $\Fq$-subplane $\bar \pi$ of $\PG(2,q^3)$, and suppose $\bar \Cplus$ meets $\li$ in the two points $\{\bar P,\bar Q\}$ (possibly repeating or in an extension). In  the exact-at-infinity Bruck-Bose representation in $\PG(6,q)$, $\V([\C])$ decomposes into a non-degenerate conic $\N_2$  (which lies in a plane of $\PG(6,q)\setminus\si$ that meets $q+1$ spread planes), together with  two planes  at infinity. Further,  we have the following. \begin{enumerate}
\item  If $\bar\C\cap\li=\{\bar P,\bar Q\}$, possibly $\bar P=\bar Q$, then $\V([\C])$ contains the two spread planes  $[P], [Q]$. Moreover, $\N_2\cap\si$ is two real points, one in $[P]$ and one in $[Q]$.
\item   If $\bar\C\cap\li=\emptyset$, then $ \bar\Cplus\cap\li=\emptyset$ and $\barCplusplus\cap \liplusplus=\{\bar P,\bar Q\}\subset\liplusplus\setminus\li$, $\bar P\neq\bar Q$. 
The variety-extension $\VCBBstarstar$ contains the two planes $\lround P\rround_{b}$ and $\lround Q\rround_{b}$ of $\sistarstar$, where $\bar b=\bar \pi\cap\li$. Moreover,  $\Ntwostarstar\cap\sistarstar$ is two  points, $K,K^{q}\in\sistarstar\setminus\sistar$, with $K$  not on a T-line. 
\end{enumerate}
\end{corollary}

\begin{proof} The proof of Theorem~\ref{conic-sec} verifies part 1 and the first statement of part 2. To prove the second statement of part 2,  
consider the non-degenerate conic $\N_2$ in $\VCBB$ and look at $\N_2\cap\si$. Recall $\alpha$ is the plane containing $\N_2$, and as $\bar\C\cap\li=\emptyset$, the line $m=\alpha\cap\si$ is exterior to $\N_2$. Hence  $\N_2$ meets $\si$ in two points $K,L$ lying in the  extension $\sistarstar\setminus\sistar$, Further, 
 $K=m\blackstar\cap\lround P\rround_b$, $L=m\blackstar\cap\lround Q\rround_b$, and 
 the points $K,L$ are conjugate with respect to the quadratic extension  of $\PG(5,q^3)$ to $\PG(5,q^6)$, that is, $L=K^{q^3}$. 
 Moreover, as $m^q=m$, $K^q\in m$ and $L=K^{q^3}=K^q$.

  Suppose $K$ lies on a T-line $\ell$, and without  loss of generality, suppose $\ell=XY^q$ for some $X,Y\in \gstar$.
  So $K^{q^2}=K\in \ell^{q^2}=X^{q^2}Y^{q^3}$. As $Y^{q^3}\in \gstar$, the plane $\langle K,g\rangle$ contains  $\ell$ and $\ell^{q^2}$, so meets $\gqstar$ and $\gqqstar$, contradicting $g,g^q,g^{q^2}$ spanning 5-space. 
  \end{proof}

\subsection{$\Fq$-conics in  an $\Fq$-subplane tangent to $\li$}

Next we look at $\Fq$-conics contained in an $\Fq$-subplane  that is tangent to $\li$.  

\begin{theorem}\Label{conic-tgt}
Let $\bar \C$ be an $\Fq$-conic in a tangent $\Fq$-subplane $\bar \pi$ of $\PG(2,q^3)$. Then  in the exact-at-infinity Bruck-Bose representation in $\PG(6,q)$, the points of   $[\C]$ form either a 6-dim nrc $\N_6$, or a 4-dim nrc $\N_4$ and a spread plane.
\end{theorem}

\begin{proof} 
Let $\bar \C$ be an $\Fq$-conic in a tangent $\Fq$-subplane $\bar \pi$ of $\PG(2,q^3)$, and let $\bar T=\bar \pi\cap\li$.
We work in the Bose representation, so $\Pit$ is a transversal plane of the Bose spread $\mathbb S$; the line at infinity $\li$ corresponds to a line $g$ of $\Pit$; and in $\Pit$, $\C$ is an $\Fq$-conic in an $\Fq$-subplane $\pi$, with $T=\pi\cap g$. Let $\cpi$ be as defined in (\ref{defcpi}), so $\cpi$ fixes the points of $\pi$, has order 3 acting on the points of $\Pit$, 
 and has order 6 when acting on the points of the quadratic extension $\Pit\blackstar$. 
Let $\Sigma_{6,q}$ be a 6-space of $\PG(8,q)$  that contains the 5-space $\si=\langle g,g^q,g^{q^2}\rangle\cap\PG(8,q)$, then the Bruck-Bose setting is constructed as  $\IBB=\IB\cap\Sigma_{6,q}$, and $g$ is a transversal line of the Bruck-Bose spread $\S$.

We want to determine the structure of the variety $\VCBB$ in the $6$-space $\Sigma_{6,q}$. First note that  $\bar \C$  contains $t\in\{q,q+1\}$ affine points, so the variety $\VCBB$ contains exactly $t$ affine points (that is, points in  $\Sigma_{6,q}\setminus\si$) and no  three of these points are collinear. So $\VCBB$ is not contained in a line. 
 By Theorem~\ref{fqconic-order6}, 
$\VCBB=\VC\cap\Sigma_{6,q}$ is generically 
 a curve of degree 6. If this is a reducible curve,  the  components may lie in an extension. 
 We will show that for our analysis, it suffices to work in either the extension to $\PG(8,q^3)$ or to $\PG(8,q^6)$, and  describe $\VCBB$ in the extension $\PG(8,q^6)$. 
By Result~\ref{conic-subplane-B},    
in $\PG(8,q^6)$,  the  points of the variety  $\VCstarstar=\Qonestarstar\cap\cdots\cap\Qninestarstar$  coincide with the points on the $q^6+1$ planes $ \{\lround X\rround_\pi\st X\in\Cplusplus\}$.
 So the extension $\VCBBstarstar=\VCstarstar\cap \Sigma_{6,q}^\blackstar$ consists of the  subspaces 
 $  \lround X\rround_\pi\cap\Sigma_{6,q}^\blackstar$, $X\in\Cplusplus$. 
 
Next we consider whether $\VCBB$ can contain a line in some extension. We argue in the extension $\PG(8,q^6)$ so that we can use our notation, however the argument works for any finite extension. Suppose $\VCBBstarstar$ contains a line $\ell$ \emph{which is not contained in $\sistarstar$}. As $\ell$ is contained in $\VCstarstar=\{\lround X\rround_\pi\st X\in\Cplusplus\}$, and each $\pi$-scroll-plane meets $\Sigma_{6,q}^\blackstar\setminus\sistarstar$ in at most one point, the line $\ell$ is not contained in a $\pi$-scroll-plane. That is, the $q^6+1$ points of $\ell$ lie one in each of the $q^6+1$ $\pi$-scroll-planes $\{\lround X\rround_\pi\st X\in\Cplusplus\}$.  This contradicts the fact that any three planes in 
  $ \{ \lround X\rround_\pi\cap\Sigma_{6,q}^\blackstar\,|\,X\in\Cplusplus\}$ span $\PG(8,q^6)$ as $\Cplusplus$ is a non-degenerate conic. We conclude that if $\VCBBstarstar$ contains a line $\ell$, then $\ell$ is contained in $\sistarstar$.  It follows that if the variety $\VCBB$ contains a plane $\alpha$ in some extension,  then $\alpha$ is contained in the hyperplane at infinity.

We now proceed on a case by case basis. In each case, we first determine the intersection at infinity, namely $ \VCBBstarstar\cap\sistarstar =
 \VCstarstar\cap\sistarstar$ using  Lemma~\ref{scroll-plane-C} which describes the  latter. We  work in  the extension to either $\PG(8,q^3)$ or $\PG(8,q^6)$ as informed by  
  Lemma~\ref{scroll-plane-C}.    
There are four cases depending on how $\bar\C$ meets $\li$. We look at each case separately, labelling  to be consistent with Result~\ref{Fqconic-cases}. 

Case \TI. Suppose 
$\bar T$ is a point of $\bar\C$, so  $\bar \Cplus\cap\li=\{\bar T,\bar Q\}$ with $\bar T\neq\bar Q$. We first look at the intersection of $\VCBB$ with $\si$. In the $\PG(8,q^3)$ Bose setting, 
consider the point $Q\in\Pit$, as  $ Q\notin\pi$, $Q$ has orbit size 3 under $\cpi$. As 
$Q\in\Cplus$, $\conjB{Q},\conjBsq{Q}\in\Cplus$. As $Q\in g$ and $g$ is not a line of $\pi$,  we have 
$\conjB{Q},\conjBsq{Q}\notin g$ and so $\conjB{Q},\conjBsq{Q}\notin \sistar$.
By Lemma~\ref{scroll-plane-C}, 
$\VCBBstar\cap\sistar =\{\lround
X\rround_\pi\cap\sistar\st X^{\cpi^i}\in\Cplus\cap g \textup{ for\ some } i=0,1,2\}$. So   to determine  $\VCBBstar\cap\sistar$, we only need to consider four planes of $\VCstar$, namely $\lround T\rround_\pi$, $\lround Q\rround_\pi$, 
 $\lround \conjB{Q}\rround_\pi$, $\lround \conjBsq{Q}\rround_\pi$. 
As $T\in\pi$,  the first plane is $\lround T\rround_\pi=[T]\star$, which  lies in $\sistar$, so lies in $\VCBBstar$. Hence the spread plane $[T]$ lies in $\VCBB$. 
As $Q\in \Pit$, the second plane is  
 $\lround Q\rround_\pi=\langle Q,(Q^{\cpis})^q, (Q^{\cpi})^{q^2}\rangle.$  
As 
$\conjB{Q},\conjBsq{Q}\notin \sistar$,  by  Lemma~\ref{cpiplus}, $\lround Q\rround_\pi$ meets $\sistar$ in one point 
  $Q$.
Similarly 
 $\lround \conjB{Q}\rround _\pi=\langle\, \conjB{Q},\ Q^q,\, (\conjBsq{Q})^{q^2}\, \rangle=\lround Q\rround_\pi^{q}$  
 meets $\sistar$ in one point $Q^q$; and  $\lround \conjBsq{Q}\rround _\pi=\langle\, \conjBsq{Q},\ (\conjB{Q})^q,\, Q^{q^2}\, \rangle=\lround Q\rround_\pi^{q^2}$ 
 meets $\sistar$ in one point $Q^{q^2}$.
  So the  variety $\VCBB$ meets $\si$ in  the spread plane $[T]$ and (in the $\Fqqq$-extension) the three conjugate points 
  $Q,Q^q,Q^{q^2}\in\sistar\setminus\si$.

 Next we look at the affine points of the variety $\VCBB$ and show they are $q$ points of a $4$-dim nrc. 
  In $\PG(2,q^3)$, let $\bar U$ be a point on  $\li$ distinct from $\bar T,\bar Q$. Let $\bar \ell$ be a line through $\bar U$ distinct from $\li$. Denote the points of $\bar \C$ by $\bar T=\bar A_0,\bar A_1,\ldots,\bar A_{q}$ and let $\bar b=\{\bar \ell\cap \bar Q\bar A_i\,|\,i=0,\ldots,q\}$. For each $i=0,\ldots,q$, the line $\bar Q\bar A_i$ is tangent to $\bar \C$, so the set of points $\bar b$ is an $\Fq$-line that meets $\li$ (it is the projection of an $\Fq$-conic onto a line).    We look at this substructure in the Bruck-Bose representation in $\Sigma_{6,q}$. By \cite[Thm 2.5]{BJ-FFA}, the affine points of $[b]$ are the affine points of a line $\ell_b$, and $\ell_b$ meets $\si$ in a point of the spread plane $[U]$. Hence  $\Sigma_4=\langle [Q],\ell_b\rangle$ is a $4$-space. Moreover, $\Sigma_4$ contains the affine points $[A_1],\ldots,[A_q]$. Let $\Sigma_5$ be a $5$-space containing $\Sigma_4$. 
  By \cite[Thm 2.7]{BJ-FFA}, the affine points of $[\pi]$ are the affine points of a scroll $\X_\pi$ that consists of $q+1$ generator lines that rule a non-degenerate conic (contained in $[T]$) and a twisted cubic (contained in a $3$-space disjoint from $[T]$); this scroll is studied in \cite{BJV52}. 
  By \cite[Thm 5.1]{BJV52}, $\Sigma_5\cap\X_\pi$ is either a $5$-dim nrc;  a $4$-dim nrc and a generator line; or a $3$-dim nrc and up to two generator lines. As $\Sigma_5\cap\X_\pi$ contains the $q$ points $[A_1],\ldots,[A_q]$ which lie in a $4$-space, $\Sigma_5\cap\X_\pi$  is not a $5$-dim nrc.  
  Suppose $\Sigma_5\cap\X_\pi$ is a $3$-dim nrc $\N$  and up to two generator lines. 
  In $\PG(2,q^3)$, the points $\bar A_1,\ldots,\bar A_q$ lie on $q$ distinct lines through $\bar T$, so in $\Sigma_{6,q}$, the points $[A_1],\ldots,[A_q]$ lie on distinct generator lines of $\X_\pi$. The $4$-space $\Sigma_4$ meets $[T]$ in a point, so contains at most one generator line of $\X_\pi$. So
$\Sigma_4\cap\X_\pi$ is a $3$-dim nrc $\N$ and at most one generator line, so $\N$ contains at least $q-1$ of the points 
$[A_1],\ldots,[A_q]$. In the $\Fqqq$-extension, the nrc $\N$ and $\si$ meet in three points. As these points lie in $\VCBBstar$,   $\Nstar$ contains the points $Q,Q^q,Q^{q^2}$. Hence $\N$ lies in a $3$-space containing the spread plane $[Q]$. In $\PG(2,q^3)$ this corresponds to a line through $\bar Q$ that contains $q-1$ of the points $\bar A_1,\ldots,\bar A_{q}$, a contradiction.
Hence $\Sigma_5\cap\X_\pi$ is a $4$-dim nrc $\N$ and a generator line $\ell$. Further, by \cite[Cor 2.8]{BJV52}, $\ell$ does not lie in the $4$-space containing $\N$. As $[A_1],\ldots,[A_q]$ is a set of $q$ points no three collinear, it follows that $\Sigma_4\cap\X_\pi$ is a $4$-dim nrc, with one point in the spread plane $[T]$, namely the point $\Sigma_4\cap[T]$. This nrc meets $\si$ in four points over some extension, and these four points lie in the extension of $\VCBB$.
  Thus the variety $\VCBB$ decomposes into the spread plane $[T]$ and a $4$-dim nrc whose 
intersection 	with $\si$ consists of one  point of $[T]$ and the  three (conjugate) points $Q,Q^q,Q^{q^2}\in\sistar\setminus\si$.

Case \TII. Suppose $\bar T\notin\C$, and that $\bar \Cplus\cap\li=\{\bar P,\bar Q\}$, $\bar P\neq \bar Q$. In the transversal plane $\Pit$, we have $P,Q\in\Pit\setminus\pi$, so $P,Q$ both have orbit size 3 under $\cpi$. 
As $P,Q\in\Cplus$,  the four distinct points $\conjB{P},\conjBsq{P}, \conjB{Q},\conjBsq{Q}$ lie in $\Cplus$. As $P,Q \in g$ and 
 $g$ is not a line of $\pi$, the four points $\conjB{P},\conjBsq{P}, \conjB{Q},\conjBsq{Q}$
do not lie on $g$, and so do not lie in $\sistar$. 
By Lemma~\ref{scroll-plane-C}, $\VCBBstar\cap\sistar =\{\lround
X\rround_\pi\cap\sistar\st X^{\mathsf c^i_\pi}\in\Cplus\cap g \textup{ for some } i=0,1,2\}$,  so we only need consider  the
six $\pi$-scroll-planes, 
   $\lround P\rround_\pi$, 
 $\lround \conjB{P}\rround_\pi$, $\lround \conjBsq{P}\rround_\pi$, $\lround Q\rround_\pi$, 
 $\lround \conjB{Q}\rround_\pi$, $\lround \conjBsq{Q}\rround_\pi$.
Exactly  one of the T-points in $\lround P\rround_\pi=\langle P,(P^{\cpis})^q, (P^{\cpi})^{q^2}\rangle,$ lies in $\sistar$, namely $P$. So by Lemma~\ref{cpiplus},  
 $\lround P\rround_\pi$ meets $\sistar$ in the single point $P$.
 Similarly, for $i=1,2,3$, 
 $\lround P^{\cpi^i}\rround_\pi=\lround P\rround_\pi^{q^i}=\langle P^{q^i}, (P^{\cpis})^{q^{i+1}},(P^{\cpi})^{q^{i+2}}\rangle $ meets $\sistar$ in one point $P^{q^i}$; and 
 $\lround Q^{\cpi^i}\rround_\pi=\lround Q\rround_\pi^{q^i}$ meets $\sistar$ in one point $Q^{q^i}$. 
 Thus   $\VCBBstar\cap\sistar$ consists of  the six points 
 $\{P,P^q,P^{q^2}, Q,Q^q,Q^{q^2}\}$.  
Hence there is no linear component at infinity, and so, as argued above, $\VCBB$ contains no plane.   It follows from Theorem~\ref{fqconic-order6} that the variety $\VCBB$ is a    curve $\K=\V^6_1$ of degree six. The $\Fq$-points of $\K$ are precisely the points of $[\C]$, that is, a set of  $q+1$ affine points, no three collinear. Moreover, the points of $[\C]$ do not lie in a plane (since a plane of $\Sigma_{6,q}$ corresponds to either a subset of a line of $\PG(2,q^3)$, or an $\Fq$-plane secant to $\li$, and $\bar\C$ does not lie in either). 

Suppose the curve $\K$ is reducible.  As $\K$ contains $q+1$ affine ($\Fq$-rational) points that do not lie in a plane, $\K$ contains an irreducible component $\A$ which is a curve of degree $x$ with $3\leq x\leq 6$. As $\K$ is reducible, $\K$ also contains a component $\B$ of degree $y$ with $0<y\leq3$. 
The curve $\B$ spans a subspace $\Sigma$ of dimension at most three. Moreover, $\Sigma$ lies in the extension $\Sigma_{6,q}^\blackstar\setminus\Sigma_{6,q}$. Hence the conjugate curves $\B^q$, $\B^{q^2}$ are distinct from $\B$, and are contained in $\K$. Adding degrees of $\A,\B,\B^q,\B^{q^2}$, we conclude that $y=1$. As argued above, if $\VCBB$ contains a line $\B$ over some extension, then $\B$ is contained in the hyperplane at infinity. However,  $\K^\star\cap\sistar=\{P,P^q,P^{q^2}, Q,Q^q,Q^{q^2}\}$, so $\K$ does not contain a line at infinity over any extension. 
We conclude that the curve $\K$ is irreducible.

Next we show that $\K$ does not lie in a $5$-space. In $\Sigma_{6,q}^\star$, the curve $\Kstar$ contains the six points $\{P,P^q,P^{q^2}, Q,Q^q,Q^{q^2}\}$, which span the $5$-space $\sistar$. As $\Kstar$ also contains an affine point,  $\Kstar$ does not lie in a 5-space. 
 Hence the curve $\Kstar$ is irreducible and does not lie in a $5$-space so by \cite[Prop 18.9]{harris}, 
  $\Kstar$  is a   6-dim nrc.  
  Hence the points of $[\C]$ form a 6-dim nrc.

Case \TIII. Suppose $\bar T\notin\C$, and that $\bar \Cplus\cap\li=\{\bar P\}$.
 A  similar argument to case \TII\ shows that $\VCBBstar\cap\sistar=\{P,P^q,P^{q^2}\}$.  
 So there is no linear component at infinity, it follows from Theorem~\ref{fqconic-order6} that the variety $\VCBB$ is a curve $\K=\V^6_1$ of degree six.
  The curve $\Kstar$ meets $\sistar$ in six points, so each point in $\{P,P^q,P^{q^2}\}$ is a repeated intersection. 
  A similar argument to case \TII\ shows that $\K$ is irreducible. 
     Suppose $\K$ lies in a 5-space $\Pi_5$  (we work to a contradiction to show this is not possible). 
  In $\PG(2,q^3)$, let $\bar U$ be a point on  $\li$ distinct from $\bar T,\bar P$. Let $\bar \ell$ be a line through $\bar U$ distinct from $\li$. Denote the points of $\bar \C$ by $\bar A_1,\ldots,\bar A_{q+1}$ and let $\bar b=\{\bar \ell\cap \bar P\bar A_i\,|\,i=1,\ldots,q+1\}$. For each $i=1,\ldots,q+1$, the line $\bar P\bar A_i$ is tangent to $\bar \C$, so the set of points $\bar b$ is an $\Fq$-line (it is the projection of an $\Fq$-conic onto a line).  Moreover, $\bar T\notin\bar \C$, so $\bar b$  is disjoint from $\li$. We look at this substructure in the Bruck-Bose representation in $\Sigma_{6,q}$. The $5$-space $\Pi_5$ contains the spread plane $[P]$ and the points $[A_i]$, $i=1,\ldots,q+1$. So $\Pi_5$ contains the $3$-spaces corresponding to the lines $\bar P\bar A_i$, and hence contains the set $[b]$. The $3$-space $[\ell]$ meets $\si$ in the spread plane $[U]$ and so is disjoint from $[P]$, hence   $[\ell]\cap \Pi_5$ is a plane. That is, $[b]$ lies in a plane. This contradicts
  \cite[Thm 2.5]{BJ-FFA} which shows that $[b]$ is a 3-dim nrc in the 3-space $[\ell]$. Hence $\K$ does not lie in a $5$-space. That is, $\K$ is an irreducible curve of degree six that is not contained in a $5$-space, so by \cite[Prop 18.9]{harris}, 
  $\K$  is a   6-dim nrc.  Hence the points of $[\C]$ form a 6-dim nrc.

Case \TIV. Suppose  $\bar \Cplus\cap\li=\emptyset$, so  $\barCplusplus\cap \liplusplus=\{\bar P,\bar Q\}\subset\liplusplus\setminus\li$, $\bar P\neq \bar Q$. By Result~\ref{Fqconic-cases}, $\bar Q=\bar P^{q^3}=\bar P^{\ocpi^3}$.
In the Bose setting in $\PG(8,q^6)$,
we have two points $P,Q\in\Pit\blackstar\setminus\Pit$, with $P,Q\in\Cplusplus\cap \gstar$ 
and  $Q=P^{q^3}=P^{\cpi^3}$.  
The six distinct points $P,P^{\cpi}, P^{\cpis}$, $Q=P^{\mathsf c^3_\pi}, P^{\mathsf c^4_\pi}, P^{\mathsf c^5_\pi}$ all lie in $\Cplusplus$, while only two of them, namely $P,Q$ lie in $\sistarstar$.  
By Lemma~\ref{scroll-plane-C}, $\VCBBstarstar\cap\sistarstar =\{\lround
X\rround_\pi\cap\sistarstar\st X^{\mathsf c^i_\pi}\in\Cplusplus\cap \gstar \textup{ for some } i=0,\ldots,5\}$,  so we only need consider  the
six $\pi$-scroll-planes, $\lround P^{\mathsf c^i_\pi}\rround_\pi$, $i=0,\ldots,5$.
As $\lround P^{\cpi^i}\rround_\pi=\lround P\rround_\pi^{q^i}=\langle P^{q^i}, (P^{\cpi^5})^{q^{i+1}},(P^{\cpi^4})^{q^{i+2}}\rangle $,  by Lemma~\ref{cpiplus}, $\lround P^{\mathsf c^i_\pi}\rround_\pi$ meets $\sistarstar$ in the point $P^{q^i}$, $i=0,\ldots,5$.
 Hence   $\VCBBstarstar\cap\sistarstar$ consists of  the six points 
 $\{P,P^q,P^{q^2}, P^{q^3},P^{q^4},P^{q^5}\}$. 
So there is no linear component at infinity, it follows from Theorem~\ref{fqconic-order6} that the variety $\VCBB$ is a curve $\K=\V^6_1$ of degree six.  A similar argument to case \TII\ shows that the curve $\Kstar$ is irreducible and does not lie in a $5$-space.   So by \cite[Prop 18.9]{harris}, 
  $\Kstar$  is a   6-dim nrc. Hence the points of $[\C]$ form a 6-dim nrc.
  \end{proof}

Full details of how  these nrcs meet $\si$ follow from the proof of Theorem~\ref{conic-tgt}.

\begin{corollary}\Label{cor:tgt}
 Let $\bar \C$ be an $\Fq$-conic in a tangent $\Fq$-subplane $\bar \pi$ of $\PG(2,q^3)$, let $\bar T={\li}\cap\bar \pi$. In  the exact-at-infinity Bruck-Bose representation in $\PG(6,q)$,  we have the following. 
\begin{enumerate}
\item   If $\bar T\in\bar \C$, then   $\bar \Cplus\cap\li=\{\bar T,\bar Q\}$ for some   $\bar  Q\neq \bar T$. The pointset of  $[\C]$ consists of the spread plane  $[T]$ and a 4-dim nrc $\N_4$. Further, $\Nfourstar\cap\sistar$ consists of one real point of $[T]$, and   the points $Q, Q^{q}, Q^{q^2}$.

\item Otherwise $\bar T\notin\bar \C$, and $\bar \Cplus$ meets $\li$ in two points, $\bar P,\bar Q$,  possibly repeated, possibly in the quadratic extension.  In this case   $\VCBB$ is a 6-dim nrc $\N_6$ of $\PG(6,q)$.  Further    $\Nsixstar$ meets $\sistar$ in    six points   $\{P,P^{q},P^{q^2},Q,Q^{q},Q^{q^2}\}$, where $P,Q\in g$,  possibly repeated, or possibly  in the extension $\gstar\setminus g$, in which case $Q=P^{q^3}$.   
\end{enumerate}
\end{corollary}

\subsection{$\Fq$-conics in  an $\Fq$-subplane exterior to $\li$}

The final case to look at is 
$\Fq$-conics contained in an $\Fq$-subplane  that is exterior to $\li$. 
We show they  correspond in the Bruck-Bose representation to either 3-dim or 6-dim nrcs, and determine how these nrcs meet $\si$. 
A coordinate based proof was given in 
 \cite{BJ-ext3}. Here we give a geometric proof, moreover, the description in Corollary~\ref{cor:ext} is a stronger result than that in \cite{BJ-ext3}.

%
%

\begin{theorem}\Label{conic-ext}
Let $\bar \C$ be an $\Fq$-conic in an exterior $\Fq$-subplane $\bar \pi$ of $\PG(2,q^3)$. 
In  the exact-at-infinity Bruck-Bose  representation in $\PG(6,q)$,  either: $\VCBB$ is a 6-dim nrc; or $\VCBB$ is a 3-dim nrc of $\PG(6,q)$ and $\VCBBstar$ consists of a 3-dim nrc and three T-lines of $\sistar$. 
\end{theorem}

\begin{proof}
Let $\bar \C$ be an $\Fq$-conic in an exterior $\Fq$-subplane $\bar \pi$ of $\PG(2,q^3)$ and let  $\ocpi$ be as defined in (\ref{defcpi}). Denote  the  $(\bar \pi,\li)$-carriers  by $\bar E,\bar E^{\ocpi},\bar E^{\ocpis}$, such that $\bar E,\bar E^{\ocpi}\in\li$. We work in the Bose representation, so the line at infinity $\li$ corresponds to a line $g$ of the transversal plane $\Pit$ of the regular 2-spread $\mathbb S$. Let $\Sigma_{6,q}$ be a 6-space of $\PG(8,q)$  that contains the 5-space $\si=\langle g,g^q,g^{q^2}\rangle\cap\PG(8,q)$, we use the Bruck-Bose setting $\IBB=\IB\cap\Sigma_{6,q}$. 

We want to determine the structure of the variety $\VCBB$ in the $6$-space $\Sigma_{6,q}$. There are four cases depending on how $\bar\C$ meets $\li$. We look at each case separately, labelling  to be consistent with Result~\ref{Fqconic-cases}.

Case \EI. Suppose  that $\bar \Cplus\cap\li=\{\bar E,\bar E^{\ocpi}\}$. 
 By Result~\ref{result-conj}, we also have $\bar E^{\ocpis}\in\bar\Cplus$.  
By Lemma~\ref{scroll-plane-C}, $\VCBBstar\cap\sistar =\{\lround
X\rround_\pi\cap\sistar\st X^{\mathsf c^i_\pi}\in\Cplus\cap g \textup{ for some } i=0,1,2\}$. As the orbit of $\bar E$ under $\ocpi$ is $\{\bar E,\bar E^{\ocpi},\bar E^{\ocpis}\}$, we 
 only need consider  the
three $\pi$-scroll-planes, $\lround E\rround_\pi$, 
$\lround E^{\cpi}\rround_\pi$, $\lround E^{\cpis}\rround_\pi$. 
First consider the plane $\lround E\rround_\pi$. Exactly two of the T-points of $\lround E\rround_\pi=\langle E,(\conjBsq{E})^q,(\conjB{E})^{q^2}\rangle=\langle E,(E^{\cpis})^q,(E^{\cpi})^{q^2}\rangle$ lie in $\sistar$, namely 
$E$ and $(E^{\cpi})^{q^2}$. Hence 
by Lemma~\ref{cpiplus}, 
$\lround E\rround_\pi\cap\sistar$ is the line $h=E(E^{\cpi})^{q^2}$. 
Similarly $\lround E^{\cpi}\rround_\pi$ and $\lround E^{\cpis}\rround_\pi$  meet $\sistar$ in the lines
$h^q$ and $h^{q^2}$ respectively. Hence 
 $\VCBB$ meets $\si$ in  the three lines $h,h^q,h^{q^2}$ of $\sistar\setminus\si$. 
Thus by Theorem~\ref{fqconic-order6}, $\VCBB$ reduces to an irreducible curve $\K=\V^3_1$ of degree three and  three lines which lie  in the extension $\sistar$.  Suppose $\K$ is contained in a plane $\alpha$ and let $\ell=\alpha\cap\si$. If $\ell$ is contained in a spread plane $[X]$, then $\K$ is contained in the $3$-space $\Sigma_3=\langle [X],\alpha\rangle$. The $3$-space $\Sigma_3$ contains a spread plane, so corresponds to a line of $\PG(2,q^3)$, implying that $\bar\C$ is contained in a line of $\PG(2,q^3)$, a contradiction. Otherwise $\ell$ meets $q+1$ spread planes, and so by \cite[Thm 2.2]{BJ-FFA}, $\alpha$ corresponds to an $\Fq$-plane of $\PG(2,q^3)$ that is secant to $\li$, implying that $\bar\C$ lies in an $\Fq$-plane secant to $\li$, a contradiction. Hence $\K$ is not contained in a plane, so by \cite[Prop 18.9]{harris}, $\K$ is a $3$-dim nrc.

Case \EII. Suppose  that $\bar \Cplus\cap\li=\{\bar P,\bar Q\}$, with  $\bar P\neq \bar Q$ and $\{\bar P,\bar Q\}\cap\{ \bar E,\bar E^{\ocpi}\}=\emptyset$.  
Similar to the proof of Case \TII\ in Theorem~\ref{conic-tgt},  
 we are interested in six planes of $\VCstar$, namely $\lround P\rround_\pi$, $\lround \conjB{P}\rround_\pi$, 
$\lround \conjBsq{P}\rround_\pi$, $\lround Q\rround_\pi$, $\lround \conjB{Q}\rround_\pi$, 
$\lround \conjBsq{Q}\rround_\pi$, moreover   each  meets $\sistar$ in a  point,  and  $\VCBB$ is a 6-dim nrc whose extension meets $\sistar$ in the six points $\{P,P^{q},P^{q^2},Q,Q^{q},Q^{q^2}\}$.  

  Case \EIII.   Suppose  that $\li$ tangent to $\Cplus$, so by Result~\ref{Fqconic-cases}, $q$ is odd and $\bar \Cplus\cap\li=\{\bar P\}$, $\bar P$ not a $(\bar\pi,\li)$-carrier. 
 The proof in this case is identical to the proof of case \TIII\ in    Theorem~\ref{conic-tgt}. The points of $[\C]$ form a 6-dim nrc that meets $\sistar$ in the points $P,P^q,P^{q^2}$, each repeated. 

Case \EIV. Suppose  that $\bar \Cplus\cap\li=\emptyset$, then by Result~\ref{Fqconic-cases}, $\barCplusplus\cap \liplusplus=\{\bar P,\bar Q\}\subset\liplusplus\setminus\li$, $\bar P\neq \bar Q$. Moreover, by Result~\ref{Fqconic-cases},   $\bar Q=\bar P^{q^3}=\bar P^{\ocpi^3}$.
So in the transversal plane $\Pit$ in $\PG(8,q^3)$, $P,Q\in \gstar\setminus g$ are points of $\Cplusplus$, and $Q=P^{q^3}=P^{\cpi^3}$.
Similar to the proof of  Case \TIV\ in Theorem~\ref{conic-tgt}, 
we are interested in six $\pi$-scroll-planes, namely $\lround P\rround_\pi$, $\lround \conjB{P}\rround_\pi$, 
$\lround \conjBsq{P}\rround_\pi$, $\lround Q\rround_\pi$, $\lround \conjB{Q}\rround_\pi$, 
$\lround \conjBsq{Q}\rround_\pi$;  each  meets $\sistarstar$ in a  point,  and   $\VCBB$ is a 6-dim nrc that meets $\sistarstar$ in the six points $\{P,P^{q},P^{q^2},P^{q^{3}},P^{q^{4}},P^{q^{5}}\}$, where $P\in\gstar\setminus g$. 
\end{proof}

We can describe in  detail how $\VCBB$ meets $\si$ in each case. 

\begin{corollary}\Label{cor:ext} Let $\bar \C$ be an $\Fq$-conic in an exterior $\Fq$-subplane $\bar \pi$ of $\PG(2,q^3)$ and 
let $\bar E,\bar E^{\ocpi}\in\li$ be  $(\bar \pi,\li)$-carriers of $\bar \pi$. In  the exact-at-infinity Bruck-Bose representation in $\PG(6,q)$  we have the following. 
\begin{enumerate}
\item If $\bar \Cplus\cap\li=\{\bar E,\bar E^{\ocpi}\}$, then    $\VCBB$ consists of a 3-dim nrc $\N_3$ of $\PG(6,q)$. Further,  $\VCBBstar\cap\sistar$ is the three T-lines  $h=E(E^{\cpi})^{q^2}$, $h^q$, $h^{q^2}$.  Further $\Nthreestar\cap\sistar=\{R,R^{q}, R^{q^2}\}$ where $R\in h$, and $R$ is not contained in an extended spread plane.

\item Otherwise  $\bar \Cplus$ meets $\li$ in two points, $\bar P,\bar Q$, which are not $(\bar\pi,\li$)-carriers,  possibly repeated, possibly in the quadratic extension. In this case   $\VCBB$ is a 6-dim nrc $\N_6$ of $\PG(6,q)$.  Further     $\Nsixstar$ meets $\sistar$ in    six points   $\{P,P^{q},P^{q^2},Q,Q^{q},Q^{q^2}\}$, where $P,Q\in g$,  possibly repeated, or possibly  in the extension $\gstar\setminus g$, in which case $Q=P^{q^3}$.   
\end{enumerate}

\end{corollary}

\begin{proof} Part 2 is proved in  the proof of Theorem~\ref{conic-ext}. Consider part 1, where $\bar \Cplus\cap\li=\{\bar E,\bar E^{\ocpi}\}$. By the proof of Case \EI\ in Theorem~\ref{conic-ext},  the variety $\VCBB$ consists of a 3-dim nrc $\N_3$ of $\PG(6,q)$, and $\VCBBstar\cap\sistar$ is the three lines $h=E(E^{\cpi})^{q^2}$, $h^q$, $h^{q^2}$. As $\bar \pi$ is exterior to $\li$, $\bar \C$ is exterior to $\li$. Hence $\VCBB$ contains $q+1$ affine points, so these  are the points of $\N_3$,  and so $\N_3$ is disjoint from $\si$. 
Hence  $\Nthreestar\cap\sistar$ is three points which are contained in $h\cup h^q\cup h^{q^2}$. 
As $\Nthreestar$ meets one of these lines, it meets   all three, that is,   let $\Nthreestar\cap h=R$, then  $R^q,R^{q^2}\in\Nthreestar$. We first show that $R$ is not a T-point.  
Suppose $R=E$, then $\langle R,R^q,R^{q^2}\rangle=\langle E,E^q,E^{q^2}\rangle$, and so the 3-space containing $\N_3$ meets $\si$ in the spread plane  $\Bo{E}=\langle E,E^q,E^{q^2}\rangle\cap\si$. However $\VCBBstar\cap\sistar$ contains exactly the lines $h,h^q,h^{q^2}$, which are not contained in the plane $\langle E,E^q,E^{q^2}\rangle$, a contradiction. So $R\neq E$, and similarly $R\neq (E^{\cpi})^{q^2}$. Hence $R$ is not a T-point.
The point $R$ lies in the T-plane $\lround E\rround_\pi$ which contains the line $h$. It is shown in \cite[Corollary 2.2]{BJW3} that two distinct  T-planes  meet in either $\emptyset$, a T-point, or a T-line, hence $R$ does not lie on an extended spread plane. 
\end{proof}

For the interested reader, we observe that 
in \cite{BJ-ext3}, 
the authors use the notation $g_{\mathbb C}=E^{\cpi} E^q$. In this article, we have 
the line $h=E(E^{\cpi})^{q^2}$. That is, the relationship between these two notations is  $h=g_{\mathbb C}^{q^2}$.  

\subsection{$\Fq$-conics using the usual Bruck-Bose convention}

As noted in Section~\ref{sec:exact}, the usual convention is that the Bruck-Bose representation is \emph{not} necessarily exact-at-infinity. That is, we do not usually include lines or planes contained in $\si$ in our description. 
So, ignoring the linear component of   $\VCBB$ (which lies in $\si$ or an extension of $\si$), we  have  the following result.

 \begin{corollary}\Label{cor:not-exactBB}
 Let $\C$ be an  $\Fq$-conic of $\PG(2,q^3)$. Then  in the Bruck-Bose $\PG(6,q)$ representation, $[\C]$ is a $k$-dim nrc, $k=2$, $3$, $4$ or $6$. 
 \end{corollary}
 
  \section{Defining $3$-special normal rational curves in $\PG(6,q)$}\Label{sec-2spec}

 We aim to characterise which nrcs of $\PG(6,q)$ correspond to $\Fq$-conics. To this end, we define the notion of a $3$-special nrc in relation to a regular 2-spread $\S$ in a hyperplane of $\PG(6,q)$. We assume throughout that $q\geq 8$, so that there is a unique 6-dim nrc through any nine points of $\PG(6,q)$, no seven in  a hyperplane.

We define a weight of a point $P\in\sistarstar$, which describes the position of $P$ in relation to the transversal lines of $\S$.
Note that this is different to the notion of weights for linear sets. In \cite{fieldreduction}, the weight of a point in a linear set relates to how the corresponding subspace meets the regular spread $\S$ of $\si$. In this article, we use weight to describe how a point in  $\sistarstar$ sits in relation to the transversal lines of the regular spread $\S$.

\begin{definition}\Label{def-good}
 Let $\S$ be a 2-regular spread in $\PG(5,q)$ embedded in $\PG(5,q^6)$. Let $\sigma\in\PGammaL(6,q^6)$ be the collineation $\,\sigma\colon X=(x_0,\ldots,x_5)\mapsto X^q=(x_0^q,\ldots,x_5^q)$. 
Let $P$ be a point in $\PG(5,q^6)$.

\begin{enumerate}
\item The {\sf weight} of $P$, denoted $w(P)$, is defined to be:
\begin{itemize}
\item $w(P)=1$ if $P$ lies in one of the (extended) transversal lines of $\S$;
\item $w(P)=2$ if $w(P)\neq 1$ and $P$ lies on a line that meets two of the  (extended) transversal lines of $\S$;
\item $w(P)=3$ otherwise. 
\end{itemize}
\item The {\sf orbit-size} of $P$, denoted $o(P)$, is the size of the orbit of $P$ under the collineation $\sigma$. 
\item Let $s(P)=1$ if $P$ lies in an extended plane of $\S$, otherwise $s(P)=2$. 
\item The point $P$ is called an {\sf $\S$-good point} if $$w(P)\times o(P)=3s(P).$$

\end{enumerate}
\end{definition}

Note   that if $P\in\PG(5,q)$, then $o(P)=1$;  if $P\in\PG(5,q^3)\setminus\PG(5,q)$, then $o(P)=3$; 
if $P\in\PG(5,q^6)\setminus\PG(5,q^3)$, 
then either $o(P)=2$ or $o(P)=6$. Moreover, if $P\in\PG(5,q^3)$ lies on one of the transversal lines $g$ of the regular 2-spread $\S$, then $o(P)=3$, and if $P\in \gstar\setminus g$, then $o(P)=6$.

\begin{definition}\Label{def-2spec} Let $\PG(6,q)$ have hyperplane at infinity $\si$ and let $\S$ be a regular 2-spread  in $\si$.  Let $\N_r$, $r\leq 6$, be an $r$-dim nrc of $\PG(6,q)$ not contained in $\si$. The curve $\N_r$ meets $\si$ in $r$ points, denoted   $\{P_1,\ldots,P_r\}$ (possibly repeated or in an extension). We say 
 $\N_r$ is    {\sf $3$-special}  with respect to $\S$ if 
 $w(P_1)+\cdots+w(P_r)=6$ and  each $P_i$ is  $\S$-good. 
\end{definition}

We next look at an $r$-dim $3$-special nrc $\N_r$ and determine the possibilities for $r$.  In each case we describe the points $P_1\ldots,P_r$, noting any relationship with the transversal lines of the regular 2-spread $\S$. 

\begin{theorem}\Label{2spec-h3} Let $\PG(6,q)$ have hyperplane at infinity $\si$ and let $\S$ be a regular 2-spread in $\si$ with transversal lines $g,g^q,g^{q^2}$. Then 
 an  $r$-dim  $3$-special nrc $\N_r$ in $\PG(6,q)$ is one of the following.
\begin{enumerate}
\item $r=6$ and $\Nsixstar\cap\sistar=\{P,P^{q},P^{q^2}, Q,Q^{q},Q^{q^2}\}$ for some $P,Q\in g$, possibly repeated. 
\item $r=6$ and $\Nsixstarstar\cap\sistarstar=\{P,P^{q},P^{q^2}, P^{q^3},P^{q^4},P^{q^5}\}$ for some $P\in \gstar\setminus g$. 

\item $r=4$ and $\N_4\cap\si=\{T\}$ and $\Nfourstar\cap\sistar=\{T,P,P^{q},P^{q^2}\}$, $P\in g$, with $T\notin\langle P,P^q,P^{q^2}\rangle$.
\item $r=3$ and $\Nthreestar\cap\sistar=\{R,R^{q},R^{q^2}\}$ where $R\in XY^{q}$, for some distinct $X,Y\in g$, $R\neq X,Y^q$.
\item $r=2$ and  $\N_2\cap\si=\{P,Q\}$ with $P,Q\in\si$ (possibly repeated). 
\item $r=2$ and  $\Ntwostarstar\cap\sistarstar=\{P,P^{q^3}\}$ with $P\in\sistarstar\setminus\sistar$, $P$ not in an extended transversal line.
 \end{enumerate}
\end{theorem}

\begin{proof} 
Let $\N_r$ be an $r$-dim $3$-special nrc in $\PG(6,q)$, and let $\K=\{K_1,\ldots,K_r\}$ denote the   points of $\N_r$ at infinity, possibly repeated, possibly in an extension. As $\N_r$ is fixed by the collineation $\sigma\in\PGammaL(6,q^6)$ with $\sigma\colon X\mapsto X^q$, the set $\K$ is fixed by $\sigma$. That is, if $P\in \K$, then $P^q,\ldots,P^{q^5}\in\K$, possibly repeated, so $o(P)$ divides 6. We consider the possibilities for $\S$-good  points in $\K$.

Suppose $A\in\K$ has $w(A)=1$. So without loss of generality, either $A\in g$, in which case $s(A)=1$ and $o(A)=3$; or $A\in \gstar\setminus g$, in which case $s(A)=2$ and $o(A)=6$. Both these possibilities satisfy $w(A)o(A)=3s(A)$, so   both give $\S$-good points. 
Suppose $B\in\K$ has $w(B)=2$. The point $B$ is $\S$-good if $2o(B)=3s(B)$, that is,  if $s(B)=2$ and $o(B)=3$. As $s(B)=2$, without loss of generality we have  $B$ is $\S$-good if $B\in XY^q$ for some $X\neq Y\in g$, $B\neq X,Y^q$
Suppose $C\in\K$ has $w(C)=3$. The point $C$ is $\S$-good if  $3o(C)=3s(C)$, so there are two possibilities.  Firstly, if $o(C)=s(C)=1$, then   $C\in\PG(5,q)$. Secondly,  if $o(C)=s(C)=2$, then $C\in \PG(5,q^6)\setminus \PG(5,q^3)$, $C$ is not in an extended transversal line, and the orbit of $C$ is $\{C, C^q\}$.  
In summary, we have  five possible point orbits for an $\S$-good point in $\K$:
\begin{enumerate}
\item[(a)] $\langle  A\rangle=\{A,A^q,A^{q^2}\}\subset\K$, $w(A)=w(A^q)=w(A^{q^2})=1$, $A\in g$;
\item[(b)] $\langle  A\rangle=\{A,\ldots,A^{q^5}\}\subset\K$,  $w(A)=\cdots=w(A^{q^5})=1$,  $A\in \gstar\setminus g$;
\item[(c)]  $ \langle  B\rangle=\{B,
B^q,B^{q^2}\}\subset\K$, $w(B)=w(B^q)=w(B^{q^2})=2$, $B\in XY^q$ for some $X,Y\in g$, $X\neq Y$, $B\neq X,Y^q$; 
\item[(d)] $\langle  C\rangle=\{C\}\subset\K$, $w(C)=3$, $C\in\PG(5,q)$; 
\item[(e)] $\langle  C\rangle=\{C,C^q\}\subset\K$, $w(C)=w(C^q)=3$, $C\in \PG(5,q^6)\setminus \PG(5,q^3)$, $C$ is not in an extended transversal line. 
\end{enumerate}

We now consider how we can combine these possible orbits of $\S$-good points to satisfy $w(K_1)+\cdots+w(K_r)=6$. 
If $r=6$, then the six points $K_1,\ldots,K_6$ are either two sets of type (a) orbits,  giving case 1 of the result;  or one   type (b) orbit, giving case 2 of the result. If $r=5$, then there is no way to combine orbits of points of type  (a),\ldots,(e)  to get five points whose weights sum to 6.
If $r=4$, then the only way to combine orbits of type  (a),\ldots,(e)  to get four points whose weights sum to 6 gives $\K=\{T,P,P^q,P^{q^2}\}$, with  $T$ type (d), and $P$ type  (a). If $T\in\langle P,P^q,P^{q^2}\rangle$, then  
the plane $\langle P,P^q,P^{q^2}\rangle$ meets $\N_4$ in four points, contradicting $\N_4$ being a 4-dim nrc.  This gives case 3 of the result. 
If $r=3$, then we need an orbit of type (c), giving case 4 of the result. 
If $r=2$, then there are two possibilities: either two points of type (d), giving case 5 of the result; or one orbit of type (e),  
 giving case 6 of the result. 
The case $r=1$ cannot occur as no point has weight 6. Hence the six cases outlined in the statement of the result are the only possibilities. 
\end{proof}

 \section{Characterising $\Fq$-conics in the Bruck-Bose representation }\Label{sec-main}
 
 The main result of this section is Theorem~\ref{thm-main}, which gives one unifying characterisation of the nrcs of $\PG(6,q)$ that correspond to $\Fq$-conics as the $3$-special nrcs.   For the remainder of this section,   $[\C]$ is a $k$-dim nrc of $\PG(6,q)$,  and we denote the unique extension of this $k$-dim nrc to 
a $k$-dim nrc of $\PG(6,q^3)$  and  $\PG(6,q^6)$ by $[\C]\star$ and $[\C]\blackstar$ respectively.

  \subsection{$\Fq$-conics} 

First we prove that an nrc of $\PG(6,q)$ corresponding to an $\Fq$-conic is $3$-special in the sense of Definition~\ref{def-2spec}.

\begin{lemma}\Label{conic-quartic-3}
Let $\bar \C$ be an $\Fq$-conic of $\PG(2,q^3)$. 
Then  in the Bruck-Bose $\PG(6,q)$ representation, $[\C]$ is a $3$-special nrc. 
\end{lemma}

\begin{proof} 
Let $\bar \C$ be an $\Fq$-conic in the $\Fq$-subplane $\bar 
\pi$ of $\PG(2,q^3)$.  We consider the three cases where $\bar \pi$ is 
secant, tangent and exterior to $\li$ separately. We label the cases to be consistent with  Result~\ref{Fqconic-cases}.

First suppose that $\li$ is a secant of $\bar\pi$. There are two 
cases to consider, depending on whether $\bar\C$ contains a point 
of $\li$.
Case \SI. If $\bar\C\cap\li=\{\bar P,\bar 
Q\}\in\bar\pi$ (possibly $\bar P=\bar Q$), then by 
 Corollary~\ref{cor:sec}   and Corollary~\ref{cor:not-exactBB}, $[\C]$ is a non-degenerate 
conic $\N_2$ that meets $\si$ in two real points, $X\in[P]$ and 
$Y\in[Q]$, possibly $X=Y$. We have $w(X)=w(Y)=3$, $o(X)=o(Y)=1$ and 
$s(X)=s(Y)=1$. Hence $X$ and $Y$ are $\S$-good and their weights sum to six, 
so $\N_2$ is  a $3$-special 2-dim nrc.

Case \SII. If $\bar\C\cap\li=\emptyset$,  by 
 Corollary~\ref{cor:sec}  and Corollary~\ref{cor:not-exactBB},
    $[\C]$ is a non-degenerate conic $\N_2$ whose extension meets $\si$ in two 
points  $K,K^q\in\sistarstar\setminus\sistar$. Further, $K,K^q$ are not on a transversal line of $\S$, or on a line that meets two transversals of $\S$. Hence    $w(K)=w(K^q)=3$. As $K,K^q\notin\sistar$, they do not lie in an extended spread plane and so $s(K)=s(K^q)=2$.
As $K,K^q$ lie on a line in the extension of $\N_2$, they lie in a quadratic extension of $\si$, so over $\sistarstar$, they have orbit size 2. Hence $K,K^q$ are $\S$-good, and their weights sum to six, so $\N_2$ is  
a $3$-special 2-dim nrc.

Now suppose  $\li$ is a tangent of $\bar\pi$, there are three cases to consider.  
Case \TI. Suppose  $\bar T=\li\cap\bar\pi\in\bar\C$. As $\li$ is not a line of $\bar\pi$, $\li$ is not a tangent of $\bar \C$. Hence   $\bar\Cplus\cap\li=\{\bar T,\bar L\}$ for some $\bar L\neq\bar T$. Then  by  Corollary~\ref{cor:tgt}   and Corollary~\ref{cor:not-exactBB},  $[\C]$ is a 4-dim nrc $\N_4$ which meets $\si$ in one real point  $Z\in[T]$ and   three points $L, L^{q}, L^{q^2}$ for some $L\in g$. 
We have $w(Z)=3$, $o(Z)=1$, $s(Z)=1$, $w(L)=1$, $o(L)=3$, $s(L)=1$, so each of the four points are $\S$-good, and their weights sum to six. Hence 
 $\N_4$ is a $3$-special 4-dim nrc. 
 
Case \TII.  Suppose $\li\cap\bar\pi\notin\bar\C$ and $\bar\Cplus\cap\li=\{\bar P,\bar Q\}$, possibly repeated. Then   by  Corollary~\ref{cor:tgt}   and Corollary~\ref{cor:not-exactBB},    $[\C]$ is  6-dim nrc  $\N_6$ whose extension  meets $\si$ 
in  points $\{P,P^{q},P^{q^2},Q,Q^{q},Q^{q^2}\}$, where $P,Q\in g$ (possibly $P=Q$). 
For $i=1,2,3$ we have $w (P^{q^i})=w (Q^{q^i})=1, $ $o(P^{q^i})=o(Q^{q^i})=3$ and $s(P^{q^i})=s(Q^{q^i})=1$.  Hence each of the six points are $\S$-good, and the weights sum to six. 
Hence 
  $\N_6$ is a $3$-special 6-dim nrc. 
  
Case \TIII. Suppose $\li\cap\bar\pi\notin\bar\C$ and $\bar\Cplus\cap\li=\emptyset$, then $\barCplusplus\cap\liplusplus=\{\bar P,\bar {Q}\}$. 
By  Corollary~\ref{cor:tgt}   and Corollary~\ref{cor:not-exactBB},    $[\C]$ is  6-dim nrc  $\N_6$ whose extension  meets $\si$ in the points $\{P,P^{q},P^{q^2},P^{q^3},P^{q^4},P^{q^5}\}$, where $P\in\gstar\setminus g$.
We have $w (P^{q^i})=1, $ $o(P^{q^i})=6$ and $s(P^{q^i})=2$ for $i=1,\ldots,6$. Hence the weights sum to six, and each $P^{q^i}$ is $\S$-good. Hence 
  $\N_6$ is a $3$-special 6-dim nrc.

Finally, suppose $\li$ is exterior to $\bar\pi$.   
Denote the $(\bar\pi,\li)$-carriers of $\bar\pi$ by $\bar E,\bar E^{\ocpi}\in\li$ and $\bar E^{\ocpis}\notin\li$. There are three cases to consider, depending on how the $\Fqqq$-conic $\bar\Cplus$ meets $\li$.

Case \EI.  Suppose $\bar\Cplus\cap\li=\{\bar E,\bar E^{\ocpi}\}$, then by  Corollary~\ref{cor:ext}   and Corollary~\ref{cor:not-exactBB}, $[\C]$ is a 3-dim nrc $\N_3$ that meets $\si$ in three points, $R,R^q,R^{q^2}\in\sistar\setminus\si$ with $R\in E(E^{\cpi})^{q}$, $R\neq E,E^{\mathsf c_{\pi}}$, $R$ not a T-point and $R$ not in an extended spread plane. So $w(R)=2$, $o(R)=3$ and $s(R)=2$, 
 hence $R$ and similarly $R^q,R^{q^2}$ are $\S$-good. So $\N_3$ is $3$-special. 

Case \EII.  Suppose $\bar\Cplus\cap\li=\{\bar P,\bar Q\}\neq\{\bar E,\bar E^{\ocpi}\}$, (possibly $\bar P=\bar Q$). Then by  Corollary~\ref{cor:ext}   and Corollary~\ref{cor:not-exactBB}, $[\C]$ is a 6-dim nrc $\N_6$, whose extension meets  $\si$ in the six points $\{P,P^{q},P^{q^2},Q,Q^{q},Q^{q^2}\}$ where $P,Q\in g$ (possibly repeated). For $i=1,2,3$ we have $w (P^{q^i})=w (Q^{q^i})=1, $ $o(P^{q^i})=o(Q^{q^i})=3$ and $s(P^{q^i})=s(Q^{q^i})=1$.  Hence each of the six points are $\S$-good, and the weights sum to six. 
Hence 
  $\N_6$ is a $3$-special 6-dim nrc. 
  
  Case \EIII.  Suppose $\bar\Cplus\cap\li=\emptyset$, so $\barCplusplus\cap\liplusplus=\{\bar P,\bar Q\}$. By  Corollary~\ref{cor:ext}  and Corollary~\ref{cor:not-exactBB}, $[\C]$ is a 6-dim nrc $\N_6$, whose extension meets  $\si$ in the six points  $\{P,P^{q},P^{q^2},P^{q^3},P^{q^4},P^{q^5}\}$, where $P\in\gstar\setminus g$. 
We have $w (P^{q^i})=1, $ $o(P^{q^i})=6$ and $s(P^{q^i})=2$ for $i=1,\ldots,6$. Hence the weights sum to six, and each $P^{q^i}$ is $\S$-good. Hence 
  $\N_6$ is a $3$-special 6-dim nrc. 
\end{proof}

\subsection{$3$-special nrcs}

If $\N_r$ is a $3$-special nrc in $\PG(6,q)$, then by Theorem~\ref{2spec-h3} there are four possibilities for $r$, namely $r=2,3,4,6$. 
 The next four lemmas look at each of these possibilities, and show that in each case $\N_r$  corresponds to an $\Fq$-conic in the corresponding Bruck-Bose plane $\PG(2,q^3)$.

\begin{lemma}\Label{l1}  In $\PG(6,q)$, let $\S$ be a regular 2-spread in the hyperplane at infinity $\si$. Let $\N_6$ be a $3$-special 6-dim nrc in $\PG(6,q)$, then $\N_6$ corresponds  to an $\Fq$-conic of  $\PG(2,q^3)$. 
%
\end{lemma}

\begin{proof} In $\PG(6,q)$, let $\S$ be a regular 2-spread in the hyperplane at infinity $\si$, and let $g,g^q,g^{q^2}$ be the three transversal lines of $\S$. 
%
 Let $\N_6$ be a $3$-special 6-dim nrc in $\PG(6,q)$, then by Theorem~\ref{2spec-h3}, parts 1 and 2,  $\N_6$ meets $\si$ in six points $\{X,X^{q},X^{q^2},\, Y,Y^{q}, Y^{q^2}\}$, and either (i) $X=Y\in g$, (ii) $X,Y \in g$ are distinct, or (iii)   $X,Y\in\gstar\setminus g$ and $Y=X^{q^{3}}$. 
The points $X,Y$ correspond in $\PG(2,q^3)$ to points  on $\li$ denoted $\bar X,\bar Y$ respectively, and either (i) $\bar X=\bar Y\in \li$, (ii) $\bar X,\bar Y \in \li$ are distinct, or (iii)  $\bar X,\bar Y$ 
  lie in the quadratic extension $\PG(2,q^6)$ of $\PG(2,q^{3})$, and they are conjugate with respect to this quadratic extension. 
  
We first show that we can choose three affine points $[A], [B],[C]$ in $\N_6$ so that in $\PG(2,q^3)$, the corresponding points $\{\bar A,\bar B,\bar C,\bar X,\bar Y\}$ are no three collinear. 
Let $[A], [B],[C]\in \N_6$, so $[A],[B],[C]$ are not collinear, and by Theorem~\ref{2spec-h3}, they are not in $\si$.
Hence  $[\alpha]=\langle\, [A],[B],[C]\,\rangle$ is a plane that meets $\si$ in a line. 
 Suppose 
 the line $[\alpha]\cap\si$ is contained  in a  plane $[Z]$ of the regular 2-spread $\S$. As $\N_6$ is a 6-dim nrc, there is an affine point $C'$ lying in $\N_6$ which is not in the 3-space $\langle\,[\alpha],[Z]\,\rangle$. Hence the plane $\langle\, [A],[B],[C']\,\rangle$ meets $\si$ in a line which is not contained in a spread plane. 
That is,  we can choose three affine points  $[A],[B],[C]$ in $\N_6$ so that the plane $[\alpha]=\langle\, [A],[B],[C]\,\rangle$ meets $\si$ in a line which is not contained in a spread plane. By  Result~\ref{BB-Baer-3}, 
in $\PG(2,q^3)$, $\bar \alpha$ is  an $\Fq$-subplane secant to $\li$. 
Suppose that  the three points $\bar A,\bar B, \bar C\in\bar \alpha$ are collinear in $\PG(2,q^3)$, then they lie on an $\Fq$-subline of $\bar \alpha$. Hence by  Result~\ref{BB-Baer-3}, in $\PG(6,q)$  the points $[A],[B],[C]$ are collinear, a contradiction. Hence the three affine points $\bar A,\bar B, \bar C$ of $\PG(2,q^3)$ are not collinear. 

We now show that the two points $\bar X,\bar Y$ are not in the secant $\Fq$-subplane $\bar \alpha$.   Suppose $\bar X\in\bar\alpha$, so $\bar X\in\PG(2,q^3)$. Then by  Result~\ref{BB-Baer-3},  in $\PG(6,q)$,  
$[\alpha]$ meets the spread plane  $[X]$ in a point. Hence $\langle\, [X],[\alpha]\,\rangle$ is a  4-space, which contains  six points of $\N_6$, namely $\{[A],[B],[C],X,X^{q},X^{q^2}\}$, contradicting $\N_6$ being a nrc. Hence $\bar X\notin\bar \alpha$ and similarly $\bar Y\notin\bar \alpha$.
That is, in $\PG(2,q^3)$,  
$\bar \alpha$ is an $\Fq$-subplane  secant to $\li$, $\bar A,\bar B, \bar C$ are non-collinear affine points of $\bar\alpha$, and  $\bar X,\bar Y$ are points on the line at infinity which are not in $\bar \alpha$. Thus the five points $\bar A,\bar B,\bar C,\bar X,\bar Y$ are no three collinear. 

We now show that in each of  the three possible cases (outlined above)  for the points $\bar X,\bar Y$, the five points $\bar A,\bar B,\bar C,\bar X,\bar Y$ lie on a unique $\Fqqq$-conic of $\PG(2,q^3)$. 
 Case (i) suppose  $\bar X=\bar Y$, then the four points $\bar A,\bar B,\bar C,\bar X$ are no three collinear, hence they    lie in a unique $\Fqqq$-conic of $\PG(2,q^3)$ that is tangent to $\li$ at the point $\bar X$. 
Case (ii) suppose   $\bar X,\bar Y\in\li$ are distinct, then as
 $\bar A,\bar B,\bar C,\bar X,\bar Y$  are five points of $\PG(2,q^3)$,  no three collinear, they   lie in a unique $\Fqqq$-conic  of $\PG(2,q^3)$. 
 Case (iii) suppose  $\bar X,\bar Y$ lie in $\PG(2,q^6)\setminus\PG(2,q^3)$, then  
 the five points $\bar A,\bar B,\bar C,\bar X,\bar Y$    lie in a unique $\mathbb F_{q^6}$-conic of $\PG(2,q^6)$. As the points $\bar X,\bar Y$ are conjugate with respect to the quadratic extension from $\PG(2,q^3)$ to $\PG(2,q^6)$, 
 this $\mathbb F_{q^6}$-conic meets $\PG(2,q^3)$ in an $\Fqqq$-conic of  $\PG(2,q^3)$.
Thus in each case, the five points $\{\bar A,\bar B,\bar C,\bar X,\bar Y\}$ lie on a unique $\Fqqq$-conic of $\PG(2,q^3)$ which we denote $\bar \O$. 
 
Similar to \cite[Lemma 5.1]{BJW1}, we can show that the three points $\bar A,\bar B, \bar C$ lie on a unique $\Fq$-conic $\bar\C$ that is contained in $\bar \O$ (that is, $\bar\Cplus=\bar\O$). Denote the $\Fq$-subplane containing $\bar\C$ by $\bar\pi$. 
 We now consider three cases depending on whether $\li$ is secant, tangent or exterior to  $\bar\pi$. In each case we either reach a contradiction, or deduce that $\N_6=[\C]$, and so $\N_6$ corresponds to an $\Fq$-conic in $\bar\pi$.

Case 1: suppose that $\bar\pi$ is secant to $\li$. There is a unique $\Fq$-subplane secant to $\li$ containing the three non-collinear points  $\bar A,\bar B,\bar C$, so $\bar \pi=\bar\alpha$. That is, $\bar \C$ is an $\Fq$-conic in a secant $\Fq$-subplane $\bar \pi=\bar\alpha$, $\bar A,\bar B,\bar C\in\bar\C$ and  $\bar X,\bar Y\notin\bar\pi$. So by  Result~\ref{Fqconic-cases}, $\bar X,\bar Y$ lie in the quadratic extension $\PG(2,q^6)\setminus\PG(2,q^3)$. 
Let $\bar b=\bar \pi\cap\li$ and let $\ocb$ be as defined in  (\ref{defcb}). 
By  Corollary~\ref{cor:sec}   and Corollary~\ref{cor:not-exactBB}, in $\PG(6,q)$,  $[\C]$ is a conic in the plane $[\alpha]=[\pi]$ and $[\C]$ meets $\si$ in two points $K,K^{q^3}\in\sistarstar\setminus\sistar$, with 
$K\in \lround X\rround_{b}$. By  Result~\ref{Fqconic-cases}, $\bar X^{\ocb}=\bar Y$, so
$\lround X\rround_b =\langle X, (X^{{{\cb}^2}})^q,(X^{\cb})^{q^2}\rangle=\langle X,X^q,Y^{q^2}\rangle$.
 As $[\pi]\blackstar$ meets the plane $\lround X\rround_b $ in the point $K$,  $\langle\, [\pi]\blackstar,\lround X\rround_{b} \rangle$ is a 4-space of $\PG(6,q^6)$ that contains six points of $\Nsixstarstar$,  namely $\{[A],[B],[C],X,X^{{{q}}},Y^{q^2}\}$, a contradiction. 
 Hence this case does not occur.

Case 2: suppose  that   $\bar\pi$ is tangent to $\li$. By  Result~\ref{Fqconic-cases}, 
 either (a) one of $\bar X$ or $\bar Y$ lies in $\bar\C$, in which case $\bar X\neq\bar Y$ or (b) neither $\bar X$ or $\bar Y$  lie in $\bar\pi$. 
 Case 2(a): suppose  that $\bar\pi\cap\li=\bar X$, so $\bar X\in \bar\C$ and $\bar Y\in \li\setminus \bar X$. Then by  Corollary~\ref{cor:tgt}   and   Corollary~\ref{cor:not-exactBB},  $[\C]$ is a 4-dim nrc that contains one point of the spread plane $[X]$, and the extension $[\C]\star$ contains the six points $\{[A],[B],[C], \, Y,Y^{q}, Y^{q^2}\}.$ Thus $[\C]\star$ and $\Nsixstar$ have these six points in common, and these six points lie in the 4-space containing $[\C]\star$, a contradiction as no six points of $\Nsixstar$ lie in  a 4-space. Hence this case does not occur.
 
  Case 2(b): suppose    $\bar X,\bar Y\notin\bar\pi$. Recall that either $X,Y\in g$ possibly equal, or $X,Y\in\gstar\setminus g$. By  Corollary~\ref{cor:tgt}   and   Corollary~\ref{cor:not-exactBB},  $[\C]$ is a 6-dim nrc whose extension contains the nine points $\{[A],[B],[C], X,X^{q},X^{q^2},\, Y,Y^{q}, Y^{q^2}\}$. That is, $[\C]$ and $\N_6$ are 6-dim nrcs whose extensions to  $\PG(6,q^3)$ or $\PG(6,q^6)$ share nine points. Hence  by \cite[Theorem 21.1.1]{H2} they are equal. 
 Hence $\N_6$ corresponds in $\PG(2,q^3)$ to the $\Fq$-conic $\bar\C$.

Case 3: suppose that  $\bar\pi$ is exterior to $\li$. There are two cases, either (a)  $\bar X,\bar Y$  are the $(\bar \pi,\li$)-carriers on $\li$, or (b) they are not. 
Case 3(a):  suppose that $\bar X=\bar E$, $\bar Y=\bar E^{{\ocpi}}$ are the $(\bar \pi,\li$)-carriers  on $\li$. Then by  Corollary~\ref{cor:ext}     and Corollary~\ref{cor:not-exactBB},  $[\C]$ is a  3-dim nrc that contains the three points  $\{[A],[B],[C]\}$, and whose extension meets $\sistar$ in three points $\{R,R^{q}, R^{q^2}\}$ where 
$R\in h=E(E^{{{\mathsf c}_\pi}})^{q^2}=XY^{q^2}.$ 
Let $\Pi_3=\langle\, [A],[B],[C],R,R^{q}, R^{q^2}\rangle$ be the 3-space of $\PG(6,q^3)$ containing the 3-dim nrc $[\C]\star$. The two  lines $h=XY^{q^2}$ and $h^{q}=X^{q}Y$ are skew and each meet $\Pi_3$ in a point. Hence   
$\Pi_5=\langle\, \Pi_3, h,h^q \rangle$ 
is a 5-space which  contains seven points of $\Nsixstar$, namely 
$\{[A],[B],[C],  X,X^{q},Y^q,Y^{q^2}\},$ 
a contradiction. Hence this case does not occur.

Case 3(b): suppose that the points $\bar X,\bar Y$ are not $(\bar\pi,\li)$-carriers. Recall either $X,Y\in g$ possibly equal, or $X,Y\in\gstar\setminus g$.  By  Corollary~\ref{cor:ext}    and Corollary~\ref{cor:not-exactBB},  $[\C]$ is a 6-dim nrc whose extension contains the nine points $\{[A],[B],[C], X,X^{q},X^{q^2},\, Y,Y^{q}, Y^{q^2}\}$. That is, the extensions of $[\C]$ and $\N_6$ to   $\PG(2,q^3)$ or $\PG(2,q^6)$ are 6-dim nrcs which share nine points. So  by \cite[Theorem 21.1.1]{H2} they are equal. 
 Hence $\N_6$ corresponds in $\PG(2,q^3)$ to   the $\Fq$-conic $\bar\C$.

In summary, the only possibilities for $\bar \pi$ are given in Cases 2b) and 3b). 
In each case we show that $\N_6$ corresponds in $\PG(2,q^3)$ to the $\Fq$-conic $\bar\C\subset\bar \pi$. That is, a $3$-special 6-dim nrc of $\PG(6,q)$ corresponds to an $\Fq$-conic of $\PG(2,q^3)$. 
\end{proof}

\begin{lemma}\Label{l2}   In $\PG(6,q)$, let $\S$ be a regular 2-spread in the hyperplane at infinity $\si$. Let $\N_4$ be a $3$-special 4-dim nrc in $\PG(6,q)$, then $\N_4$ corresponds  in $\PG(2,q^3)$  to an $\Fq$-conic contained  in a tangent $\Fq$-subplane. 
\end{lemma}

\begin{proof} In $\PG(6,q)$, let $\S$ be a regular 2-spread in the hyperplane at infinity $\si$, and let $g,g^q,g^{q^2}$ be the transversals of $\S$. We use a notation in $\PG(6,q)$ that is   consistent with our Bruck-Bose  notation, as described in the first paragraph of the proof of Lemma~\ref{l1}.

Let $\N_4$ be a $3$-special 4-dim nrc in $\PG(6,q)$, so by Theorem~\ref{2spec-h3}, $\Nfourstar$ meets $\sistar$ in four points $\{T,\, Y,Y^{q}, Y^{q^2}\}$ for some $T\in\si$,  $Y \in g$, $T\notin \langle Y,Y^q,Y^{q^2}\rangle$. The point $Y$ corresponds in $\PG(2,q^3)$ to a point $\bar  Y\in\li$. The point $T$ lies in a unique plane of the 2-spread $\S$, denote this spread plane by  $[\XT ]$, and note that as $T\notin[Y]$, we have $[\XT]\neq [Y]$. The spread plane $[\XT ]$ corresponds in $\PG(2,q^3)$  to a point $\bar \XT \in\li$, and $\bar \XT \neq \bar Y$.
Let $\Pi_4$ denote the 4-space containing $\N_4$. So $\Pi_4$ meets $\si$ in the 3-space $\langle T,[Y]\rangle$, that is, $\Pi_4\cap[\XT ]=T$. 

We first show that we can choose three affine points $[A], [B],[C]$ in $\N_4$ so that in $\PG(2,q^3)$, the five points $\bar A,\bar B,\bar C,\bar \XT ,\bar Y$ are no three collinear. 
The same argument as that in the proof of Lemma~\ref{l1} shows that we can find three affine points $[A],[B],[C]$ in $\N_4$ so that the plane $[\alpha]=\langle [A],[B],[C]\rangle$ corresponds to a secant $\Fq$-subplane $\bar \alpha$ of  $\PG(2,q^3)$, and the  three affine points $\bar A,\bar B,\bar C\in\bar \alpha$ are not collinear. 
In $\PG(6,q)$, suppose the planes $[\alpha]$ and $[\XT ]$ meet. As $[\alpha]$ is contained in $\Pi_4$ (the 4-space containing $\N_4$), $[\alpha]\cap [\XT ]=\Pi_4\cap [\XT ]=T$. That is, $[\alpha]$ contains the four points $\{[A],[B],[C],T\}$ of $\N_4$, a contradiction. Hence $[\alpha]$ does not meet $[\XT ]$. Hence a line joining any two of the points $[A],[B],[C]$ does not meet $[\XT ]$. Hence by  Result~\ref{BB-Baer-3}, 
 $\bar A,\bar B,\bar C,\bar \XT $ are four points of $\PG(2,q^3)$, no three collinear. 
Now suppose the three points $\bar A,\bar B,\bar Y$  of $\PG(2,q^3)$ are collinear. Then by  Result~\ref{BB-Baer-3},   in $\PG(6,q)$, the line joining $[A]$ and $[B]$ meets $\si$ in a point of the spread plane $[Y]$. Hence 
  $\langle [A],[B],[Y]\rangle$ is a 3-space whose extension contains five points of $\Nfourstar$, namely $\{[A],[B], Y,Y^q,Y^{q^2}\}$, a contradiction.
 So $\bar A,\bar B,\bar Y$ are not collinear,  similarly $\bar A,\bar C,\bar Y$ are not collinear and $\bar C,\bar B,\bar Y$ are not collinear. 
  Hence in $\PG(2,q^3)$,  the five points $\bar A,\bar B,\bar C,\bar \XT ,\bar Y$ are no three collinear.

As $\bar A,\bar B,\bar C,\bar \XT ,\bar Y$ are 
 five points  of $\PG(2,q^3)$, no three collinear, they  lie in a unique $\Fqqq$-conic  $\bar \O$ of $\PG(2,q^3)$. Similar to \cite[Lemma 5.1]{BJW1}, we can show that $\bar A,\bar B, \bar C$ lie on a unique $\Fq$-conic $\bar \C$  that is contained in $\bar \O$, that is $\bar\Cplus=\bar\O$. Denote the $\Fq$-subplane containing $\bar\C$ by $\bar\pi$. We consider three cases depending on whether $\li$ is secant, tangent or exterior to  $\bar\pi$. 
In each case we either reach a contradiction, or deduce that $\N_4=[\C]$ in which case $\N_4$ corresponds to the $\Fq$-conic in $\bar\pi$.

Case 1: suppose $\bar\pi$ is secant to $\li$.
As $\bar \XT ,\bar Y$ lie in $\PG(2,q^3)$, by  Result~\ref{Fqconic-cases}, they lie in $\bar\pi$, so lie in $\bar\C$. By  Result~\ref{BB-Baer-3}, in $\PG(6,q)$, $[\pi]$ is a plane of $\PG(6,q)\setminus\si$ which meets both  spread planes $[\XT ]$, $[Y]$ in a point. Hence $\langle [\pi],[Y]\rangle$ is a 4-space whose extension contains the six points $\{ [A],[B],[C], Y,Y^{q}, Y^{q^2}\}$ of $\N_4$. Hence $\N_4$ lies in this 4-space, and so $[\pi]$ meets $[\XT ]$ in the point $\Pi_4\cap[\XT ]=T$. Thus the plane $[\pi]$ contains the four points $\{ [A],[B],[C], T\}$ of $\N_4$, a contradiction. 
So this case does not occur.

Case 2: suppose  $\bar\pi$ is tangent to $\li$. Recall $\bar \XT \neq\bar Y$, so there are three cases to consider, either (a) $\bar \XT =\bar\pi\cap\li$, (b)  $\bar Y=\bar\pi\cap\li$, or (c)   $\bar \XT ,\bar Y\notin\bar\pi$. Case 2(a): suppose    that $\bar\pi\cap\li=\bar \XT $, so $\bar \XT \in \bar\C$. Then by  Corollary~\ref{cor:tgt}    and Corollary~\ref{cor:not-exactBB},  $[\C]$ is a   4-dim nrc lying in a 4-space $\Sigma_4$, and $[\C]\star$ meets $\sistar$ in the four  points $\{Y,Y^{q}, Y^{q^2},[\XT ]\cap\Sigma_4\}$.
Now  $\Pi_4$ (the 4-space containing $\N_4$) contains the three non-collinear affine points $[A],[B],[C]$ and the spread plane  $[Y]$, and so  $\Pi_4=\langle [A],[B],[C],[Y]\rangle$.
As $\Sigma_4$ contains  $[A],[B],[C]$ and $[Y]$, $\Sigma_4=\Pi_4$. Thus $[\XT ]\cap\Sigma_4=[\XT ]\cap\Pi_4=T$. Hence 
 $[\C]\star$ and $\Nfourstar$ are 4-dimensional nrcs with  seven points 
 in common, namely $\{[A],[B],[C],Y,Y^{q}, Y^{q^2},T\}$, so  by \cite[Theorem 21.1.1]{H2} they are equal. 
That is $\N_4$ corresponds to the $\Fq$-conic $\C$. 

Case 2(b): suppose  $\bar Y\in\bar\pi$, so $\bar Y\in\bar\C$. Then by  Corollary~\ref{cor:tgt}   and Corollary~\ref{cor:not-exactBB},  $[\C]$ is a   4-dim nrc contained in a 4-space $\Sigma_4$ whose extension  meets $\sistar$ in the four  points $\{\XT ,\XT ^{q}, \XT ^{q^2},\Sigma_4\cap[Y]\}.$
As the two planes $[\XT ]$ and $[\alpha]=\langle [A],[B],[C]\rangle$ lie in $\Sigma_4$, they meet in a point. This contradicts the argument above where we show that $[\alpha]$ does not meet the spread plane $[\XT ]$.
So this case does not occur. 

Case 2(c): suppose    $\bar \XT ,\bar Y\notin\bar\pi$. Recall $\XT ,Y\in g$, $\XT \neq Y$, so by  Corollary~\ref{cor:tgt}    and Corollary~\ref{cor:not-exactBB},  $[\C]$ is a   6-dim nrc, $[\C]\star$ contains the nine points $\{[A],[B],[C], \XT ,\XT ^{q},\XT ^{q^2},\, Y,Y^{q}, Y^{q^2}\}.$  That is, $[\C]\star$ and $\Nfourstar$ share six points, namely $\{[A],[B],[C],  Y,Y^{q}, Y^{q^2}\}.$ That is,  $[\C]\star$ contains six points which lie in the 4-space containing $\Nfourstar$, 
  contradicting  $[\C]\star$ being a 6-dim nrc. Hence this case cannot occur.

Case 3: suppose $\bar\pi$ is exterior to $\li$. There are two cases, either (a)  $\bar \XT ,\bar Y$  are the $(\bar \pi,\li$)-carriers on $\li$, or (b) they are not. 
Case 3(a):  suppose that $\bar \XT =\bar E$, $\bar Y=\bar E^{{\ocpi}}$ are the $(\bar \pi,\li$)-carriers  on $\li$.  Then by  Corollary~\ref{cor:ext}   and Corollary~\ref{cor:not-exactBB},  $[\C]$ is a   3-dim nrc whose extension  meets $\sistar$ in  three points  $\{R,R^{q}, R^{q^2}\}$, where 
$R\in h=\XT Y^{q^2}$, and $R\neq \XT ,Y^{q^2}$.  The line $h=XY^{q}$ lies in $\PG(6,q^3)$ and does not meet $\PG(6,q)$. The lines $h,h^q,h^{q^2}$ are transversals of a regular 2-spread $\S_h$ of $\si$. Moreover, $[\XT ],[Y]$ are planes of $\S_h$, thus  
the three planes $[\XT ],[Y]$, $\langle R,R^q,R^{q^2}\rangle\cap\si$  of $\S_h$ are pairwise disjoint. 
 The two planes $[\alpha]=\langle [A],[B],[C]\rangle$ and $\langle R,R^q,R^{q^2}\rangle\cap\si$ lie in the 3-space   containing $[\C]$. Hence $[\alpha]$ meets  $\langle R,R^q,R^{q^2}\rangle\cap\si$ in a line $m$.
As $[\alpha]$ lies in the 4-space $\Pi_4$ containing $\N_4$, the line $m$ lies in the 3-space $\Pi_4\cap\si$. This 3-space contains the plane $[Y]$, so $m$ meets $[Y]$ in at least a point. This contradicts the two planes 
$[Y]$, $\langle R,R^q,R^{q^2}\rangle\cap\si$ being disjoint. 
 Hence this case does not occur.

 Case 3(b): suppose $\bar \XT ,\bar Y$ are not   $(\bar\pi,\li)$-carriers. Recall $\XT \neq Y$ and $\XT ,Y\in g$. 
By  Corollaries~\ref{cor:ext}    and \ref{cor:not-exactBB},  $[\C]$ is a   6-dim nrc and $[\C]\star$ contains  the nine points $\{[A],[B],[C], \XT ,\XT ^{q},\XT ^{q^2},\, Y,Y^{q}, Y^{q^2}\}.$ That is,   $[\C]\star$ and $\Nfourstar$  share the six points $\{[A],[B],[C],  Y,Y^{q}, Y^{q^2}\}.$ These six points  lie in the  4-space containing $\Nfourstar$,  contradicting $[\C]\star$ being  a 6-dim nrc. Hence this case cannot occur. 

In summary, the only possibility for $\bar \pi$  is that described in case 2(a), where $\bar\pi$ is tangent to $\li$ and $\bar \XT \in\bar\pi$. Hence if $\N_4$ is a $3$-special 4-dim nrc in $\PG(6,q)$, then $\N_4$ corresponds  in $\PG(2,q^3)$  to the $\Fq$-conic $\C$, and $\bar\pi$ is tangent to $\li$. 
\end{proof}

\begin{lemma}\Label{l3}   In $\PG(6,q)$, let $\S$ be a regular 2-spread in the hyperplane at infinity $\si$. Let $\N_3$ be a $3$-special 3-dim nrc in $\PG(6,q)$, then $\N_3$ corresponds in $\PG(2,q^3)$  to an $\Fq$-conic contained in an exterior $\Fq$-subplane.
\end{lemma}

\begin{proof} In $\PG(6,q)$, let $\S$ be a regular 2-spread in the hyperplane at infinity $\si$, and let $g,g^q,g^{q^2}$ be the transversals of $\S$. We use a notation in $\PG(6,q)$ that is   consistent with our Bruck-Bose  notation, as described in the  first paragraph of the proof of Lemma~\ref{l1}. 

Let $\N_3$ be a $3$-special 3-dim nrc in $\PG(6,q)$. By Theorem~\ref{2spec-h3}, there are distinct points $X,Y\in g$ such that $\N_3\cap\si=\{R,R^{q},R^{q^2}\}$, where  $R\in XY^{q}$, $R\neq X,Y^{q}$. 

We first show that we can choose three affine points $[A], [B],[C]$ in $\N_3$ so that in $\PG(2,q^3)$, the five points $\bar A,\bar B,\bar C,\bar X,\bar Y$ are no three collinear. 
The same argument as that in the proof of Lemma~\ref{l1} shows that we can find three affine points $[A],[B],[C]$ in $\N_3$ so that the plane $[\alpha]=\langle [A],[B],[C]\rangle$ corresponds to a secant $\Fq$-subplane $\bar \alpha$ of  $\PG(2,q^3)$, and the  three affine points $\bar A,\bar B,\bar C\in\bar \alpha$ are not collinear. We next show that the points $\bar X,\bar Y$ are not in $\bar \alpha$. 
Let $\Pi_3$ be the 3-space containing $\N_3$. Then $\Pithreestar$ contains the six points $\{[A],[B],[C],R,R^{q}, R^{q^2}\}$ and 
 $\Pithreestar\cap\sistar$ is the plane $\langle R,R^{q}, R^{q^2}\rangle$.
The line $h=XY^{q}$ lies in $\PG(6,q^3)$ and does not meet $\PG(6,q)$. Hence the three lines 
 $h$, $h^q,h^{q^2}$ are the transversal lines of a regular 2-spread $\S_h$ of $\PG(6,q)$. The two regular 2-spreads $\S,\S_h$ share exactly two planes, namely $[X],[Y]$. As $\S_h$ contains the three planes 
 $[X],[Y]$ and  $\langle R,R^{q}, R^{q^2}\rangle\cap\si$, these three planes are pairwise disjoint. Hence $\Pi_3$ does not meet $[X]$ or $[Y]$ and so the plane $[\alpha]$ does not meet $[X]$ or $[Y]$. 
Thus in $\PG(2,q^3)$, $\bar A,\bar B,\bar C$ are non-collinear affine points in the secant $\Fq$-subplane $\bar \alpha$,  and the points $\bar X,\bar Y$ lie in $\li\setminus \bar \alpha$. Hence 
the five points $\bar A,\bar B,\bar C,\bar X,\bar Y$ are no three collinear.

As $\bar A,\bar B,\bar C,\bar X,\bar Y$  are  five points of $\PG(2,q^3)$,  no three collinear,  they lie in a unique $\Fqqq$-conic  $\bar \O$  of $\PG(2,q^3)$. Similar to \cite[Lemma 5.1]{BJW1}, we can show that the three points $\bar A,\bar B, \bar C$ lie on a unique $\Fq$-conic $\bar\C$ that is contained in $\bar \O$ (that is, $\bar \Cplus=\bar \O$). Denote the $\Fq$-subplane containing $\bar\C$ by $\bar\pi$. We consider three cases depending on whether  $\bar\pi$ is secant, tangent or exterior to  $\li$. 
In each case we either reach a contradiction, or deduce that $\N_3=[\C]$, and so $\N_3$ corresponds to an $\Fq$-conic in $\bar\pi$.

Case 1: suppose $\bar\pi$ is secant to $\li$. 
In this case we have $\bar\C\cap\li=\emptyset$ and $\bar\Cplus\cap\li=\{\bar X,\bar Y\}\subset\PG(2,q^3)$. This contradicts  Result~\ref{Fqconic-cases}  which shows that if  $\bar\C$ lies in a secant $\Fq$-subplane and $\bar\C\cap\li=\emptyset$, then the  intersection of $\Cplus$ with $\li$ lies in $\PG(2,q^6)\setminus\PG(2,q^3)$. 
Hence this case cannot occur.

Case 2: suppose $\bar\pi$ is tangent to $\li$. Recalling that $\bar X\neq \bar Y$, there are two cases to consider. Either (a) exactly one of $\bar X,\bar Y$ lie in $\bar \pi$ or (b) $\bar X,\bar Y\notin\bar \pi$.
Case 2(a): suppose $\bar X\in\bar\pi$, then $\bar X\in\bar\C$. By  Corollary~\ref{cor:tgt}   and Corollary~\ref{cor:not-exactBB}, $[\C]$ is a 4-dim nrc whose extension contains  the six points $\{[A],[B],[C],Y,Y^{q},Y^{q^2}\}$ and a point in the spread plane $[X]$. So the 4-space $\Pi_4$ containing $[\C]$ contains the two planes $[Y]$ and $[\alpha]=\langle [A],[B],[C]\rangle$, so these planes meet in at least a point, a contradiction as the two planes $[Y]$, $[\alpha]$ do not meet. Hence this case cannot occur. 
 Similarly the case $\bar Y\in\bar \pi$ cannot occur.

Case 2(b): suppose $\bar X,\bar Y\notin\bar\pi$. Recall $X\neq Y$ and $X,Y\in g$. By  Corollary~\ref{cor:tgt}   and Corollary~\ref{cor:not-exactBB}, $[\C]$ is a 6-dim nrc and $[\C]\star$ contains the nine points $\{[A],[B],[C],X,X^{q},X^{q^2},Y,Y^{q},Y^{q^2}\}.$ 
Let  $\Pi_3$ be the 3-space containing $\N_3$. So $\Pithreestar$ contains the six points $\{[A],[B],[C],R,R^{q},R^{q^2}\}$ where $R\in h=XY^{q^2}$. As $h,h^q$ are skew, and are not contained in $\Pithreestar$, $\langle \Pithreestar,h,h^q\rangle$ is a 5-space  of $\PG(6,q^3)$ that contains seven points of the 6-dim nrc $[\C]\star$, namely $\{[A],[B],[C],X,X^{q},Y,Y^{q^2}\}$,
a contradiction. Hence this case cannot occur.

Case 3: suppose $\bar\pi$ is exterior to $\li$. There are two cases, either (a)  $\bar X,\bar Y$  are the $(\bar \pi,\li$)-carriers on $\li$, or (b) they are not. 
Case 3(a):  suppose that $\bar X=\bar E$, $\bar Y=\bar E^{{\ocpi}}$ are the $(\bar \pi,\li$)-carriers  on $\li$.  By  Corollary~\ref{cor:ext}    and Corollary~\ref{cor:not-exactBB}, $[\C]$ is a 3-dim nrc in a 3-space $\Sigma_3$ whose extension contains the six points $\{[A],[B],[C] ,K,K^q,K^{q^2}\}$, where $K\in h=XY^{q^2}$, $K\neq X,Y^{q^2}$.
Recall that the 3-space  $\Pithreestar$  containing $\Nthreestar$  contains the six points $\{[A],[B],[C],R,R^{q},R^{q^2}\}$ where $R\in h=XY^{q^2}$.  We show that $R=K$. 
Let $\S_h$ be the regular 2-spread of $\si$ with transversal lines $h,h^q,h^{q^2}$. 
The  planes $\langle K,K^{q},K^{q^2}\rangle\cap\si=\Sigma_3\cap\si$ and $\langle R,R^{q},R^{q^2}\rangle\cap\si=\Pi_3\cap\si$ are planes of $\S_h$, so are either equal or disjoint. The plane $[\alpha]=\langle [A],[B],[C]\rangle$ lies in both $\Sigma_3$ and $\Pi_3$, so $\Sigma_3$ and $\Pi_3$ have a non-empty intersection in $\si$.
That is, the two planes $\langle K,K^{q},K^{q^2}\rangle$, $\langle R,R^{q},R^{q^2}\rangle$ are not disjoint, so they are equal and hence
 $K=R$. Hence 
 $[\C]$ and $\N_3$ are 3-dim nrcs whose extensions have   six points  
in common, namely $\{[A],[B],[C] ,R,R^q,R^{q^2}\}$.
 Hence by \cite[Theorem 21.1.1]{H2} they are equal. 
Case 3(b): suppose $\bar X,\bar Y$ are not  $(\bar\pi,\li)$-carriers, then a similar argument to that in Case  2(b) gives a contradiction, and so this case does not occur. 

In summary, the only possibility is that described in case 3(a). Thus if $\N_3$ is a $3$-special 3-dim nrc in $\PG(6,q)$, then $\N_3$ corresponds in $\PG(2,q^3)$  to an $\Fq$-conic contained in an exterior $\Fq$-subplane.
\end{proof}

\begin{lemma}\Label{l4}  In $\PG(6,q)$, let $\S$ be a regular 2-spread in the hyperplane at infinity $\si$. Let $\N_2$ be a $3$-special 2-dim nrc in $\PG(6,q)$, then $\N_2$ corresponds  in $\PG(2,q^3)$ to an $\Fq$-conic contained in an $\Fq$-subplane secant   to $\li$. 
\end{lemma}

\begin{proof} In $\PG(6,q)$, let $\S$ be a regular 2-spread in the hyperplane at infinity $\si$.  By  Theorem~\ref{conic-sec}, an $\Fq$-conic in a secant $\Fq$-subplane corresponds to a 2-dim nrc. Let $\N_2$ be a $3$-special 2-dim nrc in $\PG(6,q)$. 
As there is a direct isomorphism between secant $\Fq$-subplanes of $\PG(2,q^3)$  and planes of $\PG(6,q)\setminus \si$ that meet $\si$ in a line not contained in a spread plane, $\N_2$ corresponds to an $\Fq$-conic of a secant $\Fq$-subplane.
\end{proof}

\subsection{Proof of Theorem~\ref{thm-main}}

 Let $\bar\C$ be an $\Fq$-conic of $\PG(2,q^3)$. Then by
Lemma~\ref{conic-quartic-3}, $[\C]$ is a  $3$-special nrc of $\PG(6,q)$. Conversely, 
 let $\N_r$ be a $3$-special nrc in $\PG(6,q)$. By Theorem~\ref{2spec-h3} there are four possibilities for $r$, namely $2,3,4,6$. Lemmas~\ref{l1}, \ref{l2}, \ref{l3}, and \ref{l4} show that in each of these cases, $\N_r$ corresponds to an $\Fq$-conic in $\PG(2,q^3)$. 
This completes the proof of Theorem~\ref{thm-main}.\hfill$\square$

\section{Conclusion}\Label{conclude}

It seems likely that a similar characterisation holds for 
$\Fq$-conics of $\PG(2,q^n)$ in the Bruck-Bose representation in $\PG(2n,q)$, using a regular $(n-1)$-spread denoted $\S$. However some of the details in the techniques used in the article do not generalise easily, and it seems likely that a different approach is needed. 
 
 In particular, let $\S$ denote the Bruck-Bose regular $(n-1)$-spread in $\si$, the hyperplane at infinity of $\PG(2n,q)$. 
We conjecture that
a $k$-dim normal rational curve $\N_k$ in $\PG(2n,q)$ corresponds via the Bruck-Bose representation to  an  $\Fq$-conic of $\PG(2,q^n)$ if and only if $\N_k$ is $n$-special with respect to $\S$. The definition of $n$-special is conjectured to be a generalisation of Definitions~\ref{def-good} and \ref{def-2spec}, including the following: for a point $P$ in an extension of $\si$,  $w(P)=i$ if $i$ is the smallest integer such that $P$ lies in an $(i-1)$-space that meets $i$ transversal lines of $\S$; a point $P$ is $\S$-good if $w(P)\times o(P)=n\times s(P)$; and $\N_k$ is $n$-special if  the $k$ points of $\N_k$ at infinity are $\S$-good, and their 
  weights  sum to $2n$.

\bigskip\bigskip

{\bfseries Author information}

S.G. Barwick. School of Mathematical Sciences, University of Adelaide, Adelaide, 5005, Australia.
susan.barwick@adelaide.edu.au

W.-A. Jackson. School of Mathematical Sciences, University of Adelaide, Adelaide, 5005, Australia.
wen.jackson@adelaide.edu.au

P. Wild. Royal Holloway, University of London, TW20 0EX, UK. peterrwild@gmail.com

\end{document}